\documentclass[a4paper,reqno]{article}
\usepackage{amsmath,amsfonts,amssymb,amsthm,amsrefs,bbm,graphicx,enumerate,geometry}
\usepackage{mathrsfs}

\usepackage{subfigure} 

\usepackage{color}

\usepackage{framed} 
\usepackage{hyperref}  
\hypersetup{
    linktoc=page,
    linkcolor=red,          
    citecolor=blue,        
    filecolor=blue,      
    urlcolor=cyan,
   colorlinks=true           
}

\numberwithin{equation}{section}

\newtheorem{thm}{Theorem}[section]

\newtheorem{lemma}[thm]{Lemma}

\newtheorem{cor}[thm]{Corollary}
\newtheorem{defn}[thm]{Definition}

\newtheorem{remark}[thm]{Remark}

\newtheorem{conjecture}[thm]{Conjecture}

\renewcommand{\epsilon}{\varepsilon}

\def\<#1{\langle #1\rangle}

\begin{document}{\allowdisplaybreaks[4]}


\title{Multiple chordal SLE(0) and classical Calogero-Moser system}
\author{
    Jiaxin Zhang
    \footnotemark[1]
   }
\renewcommand{\thefootnote}{\fnsymbol{footnote}}

\footnotetext[1]{{\bf zhangjx@caltech.edu} Department of Mathematics, California Institute of Technology}

\maketitle

\begin{abstract}
We develop a general theory of multiple chordal $\mathrm{SLE}(0)$ systems of type $(n, m)$ for positive integers $n$ and $m$ with $m \leq \lfloor n/2 \rfloor$, extending the construction of~\cite{ABKM20} beyond the previously studied case $n = 2m$.

By applying integrals of motion for the Loewner dynamics, we show that, in the $\mathbb{H}$-uniformization with the marked point $q = \infty$, the traces of type $(n, m)$ multiple chordal $\mathrm{SLE}(0)$ systems correspond to the real locus of real rational functions with $n$ real simple critical points, $m$ simple poles, and a pole of order $n - 2m + 1$ at $\infty$.

The stationary relations connect the classification of multiple radial SLE(0) systems to the enumeration of critical points of the master function of rational Knizhnik-Zamolodchikov (KZ) equations. 

Furthermore, we demonstrate that, under a common capacity parametrization, the Loewner dynamics evolve according to the classical Calogero-Moser Hamiltonian.

\end{abstract}

\tableofcontents

\newpage

\section{Introduction}
\subsection{Background}
\indent 

The Schramm--Loewner evolution $\mathrm{SLE}(\kappa)$ is a one-parameter family of conformally invariant random curves that arise as scaling limits of interfaces in two-dimensional critical lattice models~\cite{Sch00}. For $\kappa > 0$, these processes are uniquely characterized by the domain Markov property and conformal invariance, and are closely related to conformal field theories (CFTs) with central charge
\[
c(\kappa) = \frac{(3\kappa - 8)(6 - \kappa)}{2\kappa}.
\]
The relationship between $\mathrm{SLE}(\kappa)$ and CFT, often referred to as the SLE/CFT correspondence~\cite{Car96, FK04, BB03a, FW03, Dub15a, Pel19}, provides a framework for constructing martingale observables and partition functions using conformal field theory methods.

Multiple $\mathrm{SLE}$ systems describe families of non-intersecting random curves with prescribed pairwise connections among boundary points. The best-understood case is the type $(2n, n)$ chordal configuration, where $2n$ boundary points are joined by $n$ disjoint curves in the upper half-plane. This setting has received considerable attention in the study of multiple $\mathrm{SLE}(\kappa)$ systems and their deterministic limit as $\kappa \to 0$, as shown in~\cite{Dub06, KL07, FK15a, PW19}. These foundational results have paved the way for various extensions and generalizations.

In this paper, we generalize the construction of multiple chordal $\mathrm{SLE}(0)$ systems beyond the $(2n, n)$ case by introducing configurations of type $(n, m)$. Such a configuration consists of $n$ ordered boundary points, together with a marked point $u \in \partial\Omega$, where $m$ disjoint, non-crossing connections are formed among the $n$ points, and the remaining $n - 2m$ points are connected to $u$. These generalized link patterns provide a natural setting for the extension of deterministic multiple SLE theory, and lead to new analytic and geometric structures, including connections with rational functions and classical integrable systems.

\subsection{Multiple chordal SLE($0$) systems}
\indent

For $\kappa = 0$, we derive the commutation relations for multiple chordal SLE systems. Unlike the case $\kappa > 0$, these relations do not imply the existence of a partition function $\psi$. See Section~\ref{commutation when kappa=0} for details. This distinction shows that not all chordal SLE(0) systems can be quantized or viewed as classical limits of random multiple SLE($\kappa$) processes.

We introduce a class of multiple chordal SLE(0) systems defined by stationary relations, which emerge from analyzing the $\kappa \to 0$ limit of multiple SLE($\kappa$) partition functions. Extending the methods of~\cite{ABKM20}, we develop a theory of multiple chordal SLE(0) systems of general type $(n,m)$.

These systems are characterized by configurations of growth points $\boldsymbol{x} = \{x_1, x_2, \ldots, x_n\}$ and a set of marked points $\boldsymbol{\xi} = \{\xi_1, \xi_2, \ldots, \xi_m\}$ known as screening charges, which satisfy a system of stationary relations. The stationary relations are equivalent to the critical point equations for a master function $\Phi(\boldsymbol{x}, \boldsymbol{\xi})$, which is also known as the master function for the rational Knizhnik--Zamolodchikov equation.

\begin{defn}[Stationary relations]
\label{Stationary relations}
Let $\boldsymbol{x}$ and $\boldsymbol{\xi}$ defined as above, the stationary relations between growth points and screening charges are given by
\begin{equation}
-\sum_{j=1}^{n} \frac{2}{\xi_k-x_j}+\sum_{l \neq k} \frac{4}{\xi_k-\xi_l}+ \frac{2n-4m+4}{\xi_k-u}=0, k =1,2,\ldots,n
\end{equation}

when $u=\infty$, the stationary relations are simplified as 
\begin{equation}
-\sum_{j=1}^{n} \frac{1}{\xi_k-x_j}+\sum_{l \neq k} \frac{2}{\xi_k-\xi_l}=0, k =1,2,\ldots,n
\end{equation}

\end{defn}

Based on the stationary relations, we formulate the definition of the multiple chordal $\mathrm{SLE}(0)$ Loewner chain.

\begin{defn}[Multiple chordal SLE(0) Loewner chain] \label{multiple chordal SLE(0) Loewner chain via stationary relations}
Given growth points $\boldsymbol{x}=\{x_1,x_2,\ldots,x_n\}$ on the unit circle, and a marked interior point $u=0$, and screening charges $\{\xi_1,\xi_2,\ldots,\xi_m\}$ involution symmetric and solving the \textbf{stationary relations}. 

Let $\boldsymbol{\nu}=\left(\nu_1, \ldots, \nu_n\right)$ be a set of parametrization of the capacity, where each $\nu_i:[0, \infty) \rightarrow[0, \infty)$ is assumed to be measurable. 

In the upper half plane $\Omega=\mathbb{H}, u = \infty$. We define the multiple chordal SLE(0) 
Loewner chain as a random normalized conformal map $g_t=g_t(z)$, with $g_{0}(z)=z$ whose evolution is described by the Loewner differential equation:
\begin{equation}
\partial_t g_t(z)=\sum_{j=1}^n \frac{2\nu_j(t)}{g_t(z)-x_j(t)}, \quad g_0(z)=z,
\end{equation}
and the driving functions $x_j(t), j=1, \ldots, n$, evolve as

\begin{equation}\label{multiple SLE 0 driving}
\dot{x}_j= \nu_j(t) \frac{\partial {\rm log} \mathcal{Z}(\boldsymbol{x},\boldsymbol{\xi})}{\partial x_j}+\sum_{k \neq j} \frac{2\nu_k(t)}{x_j-x_k}
\end{equation}
where
\begin{equation}
    \mathcal{Z}(\boldsymbol{x}, \boldsymbol{\xi})= \prod_{1 \leq i<j \leq n}(x_i-x_j)^{2} \prod_{1 \leq i<j \leq m}(\xi_i-\xi_j)^{8} \prod_{i=1}^{n} \prod_{j=1}^m\left(x_i-\xi_j\right)^{-4}
\end{equation}
The logarithmic derivative of $\mathcal{Z}(\boldsymbol{x}, \boldsymbol{\xi})$ with respect to $x_j$ (treating $x$ and $\xi$ as independent variables) is given by:
\begin{equation}
   \frac{\partial \mathcal{Z}(\boldsymbol{x},\boldsymbol{\xi})}{\partial x_j}=\sum_{k\neq j} \frac{2}{x_j-x_k}-2\sum_{l} \frac{2}{x_j-\xi_l}
\end{equation}
for $j=1,\ldots,n$,
The flow map $g_t$ is well-defined up to the first time $\tau$ at which $x_j(t)=x_k(t)$ for some $1 \leq j<$ $k \leq n$. For each $z \in \mathbb{C}$, the process $t \mapsto g_t(z)$ is well-defined up to the time $\tau_z \wedge \tau$, where $\tau_z$ is the first time at which $g_t(z)=z_j(t)$. The hull associated with this Loewner chain is denoted by

$$
K_t=\left\{z \in \overline{\mathbb{H}}: \tau_z \leq t\right\}
$$

\end{defn}

In order to characterize the traces of the multiple chordal SLE(0) systems, we introduce a 
class of rational functions $\mathcal{CR}_{m,n}(\boldsymbol{x})$ with prescribed zeros $\boldsymbol{x}=\{x_1,x_2,\ldots,x_n\}$. 

\begin{defn}
Let $\boldsymbol{x} = \{x_1, x_2, \ldots, x_n\}$ be distinct points on the unit circle. Define $\mathcal{CR}_{m,n}(\boldsymbol{x})$ as the class of real rational functions \( R(z) \) satisfying the following properties:

\begin{enumerate}
    \item  
    The rational function \( R(z) \) is symmetric under conjugation, i.e.,  
    \[
    \overline{R\left(z^*\right)} = R(z).
    \]

    \item  
    The points \( \{x_1, x_2, \ldots, x_n\} \) are the distinct simple critical points of \( R(z) \).

    \item   
    The rational function \( R(z) \) has distinct simple poles at \( \{\xi_1, \xi_2, \ldots, \xi_m\} \).

    \item   
    The rational function \( R(z) \) has a pole of order \( n - 2m+ 1 \) at infinity, ensuring the total difference between the number of zeros and poles is \( 0 \).
\end{enumerate}
\end{defn}

\begin{thm} \label{stationary relations rational functions}
 The following are equivalent:
\begin{itemize}
    \item[\rm(1)]$\boldsymbol{\xi}$ are conjugation symmetric 
 and $\boldsymbol{x}$ on the real line satisfying the \textbf{stationary relations}  
    \item[\rm(2)]There exists
$R(z)\in \mathcal{CR}_{m,n}(\boldsymbol{z})$ with  $\boldsymbol{\xi}$ as poles and $\boldsymbol{x}$ as critical points.

\end{itemize}

\end{thm}

In contrast to the radial setting, where horizontal trajectories of quadratic differentials characterize the curves, we show that the traces of the chordal SLE(0) system are given by the real locus of certain rational functions $R(z)$ belonging to a space $\mathcal{CR}_{m,n}(\boldsymbol{x})$.

By constructing the field integral of motion for multiple chordal SLE(0) systems, we demonstrate that the real locus of $R(z)\in \mathcal{R}_{m,n}(\boldsymbol{x})$ starting at the growth points $\boldsymbol{x}$ are exactly the traces of the multiple chordal SLE(0) processes with growth points $\boldsymbol{x}$ and screening charges $\boldsymbol{\xi}$.

\begin{thm}\label{traces as real locus}
    Let $x=\{x_1,x_2,\ldots,x_{n}\}$ be distinct points on the real line, and screening charges $\xi =\{\xi_1,\xi_2,\ldots,\xi_{m}  \}\subset \mathbb{C}$ closed under conjugation and solve the stationary relations
\begin{equation}
-\sum_{j=1}^{n} \frac{1}{\xi_k-x_j}+\sum_{l \neq k} \frac{2}{\xi_k-\xi_l}=0, k =1,2,\ldots,m
\end{equation}

Then there exists an $R \in \mathcal{CR}_{m,n}(\boldsymbol{x})$ with $\boldsymbol{\xi}$ as finite poles and $\boldsymbol{x}$ as finite critical points. The hulls $K_t$ generated by multiple Loewner flows with parametrization $\boldsymbol{\nu}(t)$ are a subset of $\Gamma(R)$, up to any time $t$ before the collisions of any poles or critical points. Up to any such time
$$
(R(z)) \circ g_t^{-1} \in \mathcal{CR}_{m,n}(\boldsymbol{x}(t)) 
$$
where $\boldsymbol{x}(t)$ is the location of the critical points at time $t$ under the multiple chordal Loewner flow with parametrization $\boldsymbol{\nu}(t)$.
\end{thm}

The key ingredient in the proof is to construct an integral of motion for the multiple chordal Loewner flows.
\begin{thm}\label{integral of motion in H}
Let $x_1,x_2,\ldots,x_n$ be distinct points and $u$ a marked point on the real line, for each $z \in \overline{\mathbb{H}}$
\begin{equation}
N_t(z)=(g_t(z)-u)^{2m-n-2}g'_{t}(z)\frac{\prod_{k=1}^{n}(g_t(x)-z_k(t))}{\prod_{j=1}^{m}(g_t(z)-\xi_j(t))^2}
\end{equation}
is an integral of motion on the interval $[0,\tau_t \wedge \tau)$ for the multiple chordal Loewner flow with parametrization $\nu_j(t)$.

If we set $u=\infty$, the field integral of motions degenerates to
\begin{equation}
N_t(z)=g'_{t}(z)\frac{\prod_{k=1}^{n}(g_t(z)-z_k(t))}{\prod_{j=1}^{m}(g_t(z)-\xi_j(t))^2}
\end{equation}
\end{thm} 

\begin{remark}
    \( N_t(z) \) is a field integral of motion for arbitrary initial positions of screening charges \( \boldsymbol{\xi} \) even without assuming stationary relations. The stationary relations imply the existence of a quadratic differential \( R(z) \in \mathcal{CR}(\boldsymbol{z}) \), see Theorem \ref{traces as real locus}.
\end{remark}

\begin{remark}
    The integral of motion is motivated by a martingale observable in conformal field theory. For a field \( \mathcal{X} \) in the OPE family \( F_{\beta} \),
    \[
    \hat{\mathbf{E}}[\mathcal{X}] := \frac{\mathbf{E}[\mathcal{X} \mathcal{O}_{\beta}]}{\mathbf{E}[\mathcal{O}_{\beta}]}
    \]
    is a martingale observable where \( \mathcal{O}_{\beta} \) is a vertex field. In our situation, we choose \( \mathcal{X} \) to be the chiral vertex field and take the classical limit as \( \kappa \to 0 \). The martingale observable degenerates to the integral of motion. We will discuss the construction of the field \( \mathcal{X} \) in Section \ref{Multiple Chordal Martingale Observable}.
\end{remark}

Based on the previous theorems, we have already established the following classification result:

\begin{thm} 
The classification of the following three objects is equivalent:
\begin{itemize}
    \item The multiple chordal SLE(0) system with screening charges \( \boldsymbol{\xi} \) and growth points \( \boldsymbol{x} \) that solve the stationary relations.
    \item Rational functions \( \mathcal{R}(\boldsymbol{z}) \) with \( \boldsymbol{\xi} \) as poles and \( \boldsymbol{z} \) as critical points, whose real locus form a \( (n,m) \) chordal link pattern.
    \item The critical point of the master function \( \Phi(\boldsymbol{z},\boldsymbol{\xi}) \) (also known as the master function for the rational Knizhnik–Zamolodchikov equations), where
    \[
    \Phi(\boldsymbol{x}, \boldsymbol{\xi}) =  \prod_{1 \leq i < j \leq n} (x_i - x_j)^{2} \prod_{1 \leq i < j \leq m} (\xi_i - \xi_j)^{8} \prod_{i=1}^{n} \prod_{j=1}^{m} (x_i - \xi_j)^{-4} 
    \]
\end{itemize}  
\end{thm}

The enumeration problem for these rational functions and the critical points of the master functions has been resolved in \cite{EG02, SV03}. For a detailed explanation, see Section~\ref{Enumerative geometry of chordal multiple SLE(0)}.
\subsection{Relations to classical Calogero-Moser systems}
\indent

From a Hamiltonian perspective, we show that multiple chordal SLE$(0)$ Loewner evolutions with uniform capacity parametrization (i.e., $\nu_j(t) = 1$), in the presence of $m$ screening charges, still correspond to a special case of the classical Calogero--Moser system.

The \textbf{stationary relations} can be interpreted as initial conditions for the particles and $n$ quadratic null vector equations as $n$ null vector Hamiltonians, which are related to the classical Calegro-Moser Hamiltonian via the lax pair.
Furthermore, these null vector Hamiltonians induce commuting Hamiltonian flows along the submanifolds defined as the intersection of their level sets.

In the following theorem, we establish a general result to describe the dynamics of multiple chordal SLE(0) systems.
\begin{thm}
\label{CM results kappa=0}
From the dynamical point of view, the multiple chordal SLE(0) systems are described by $U_j$, 
$$\dot{x}_j = U_j(x)+\sum_{k\neq j}\frac{2}{x_j-x_k}$$
where $U_j$ satisfies the quadratic null vector equations.
\begin{equation}
\frac{1}{2}U^2_j+ \sum_{k\neq j}\frac{2}{x_k-x_j}U_k-\sum_{k\neq j}\frac{6}{x_j-x_k}= 0
\end{equation} 
\begin{itemize}

\item[(i)] We transform the multiple chordal SLE(0) system into a Calogero-Moser system by introducing the momentum function $p_j$,
\begin{equation}
p_j=\left(U_j+\sum_{k \neq j}\frac{2}{x_j-x_k} \right)   
\end{equation}
then $p_j$ satisfies the null vector equation Hamiltonian:
\begin{equation}
\mathcal{H}_j(\boldsymbol{x}, \boldsymbol{p})=\frac{1}{2}p_j^2-\sum_{k \neq j}\left(p_j+p_k\right) f_{j k}+ \sum_k \sum_{l \neq k} f_{j k} f_{j l}-2\sum_{k\neq j} f_{j k}^2  = 0
\end{equation}

The sum of the null vector Hamiltonian $\mathcal{H}_j$ is exactly the classical Calogero-Moser Hamiltonian.

\begin{equation}
\mathcal{H}=\sum_j \mathcal{H}_j = \sum \frac{p_{j}^{2}}{2}-\sum_{1\leq j< k \leq n}\frac{8}{(x_j-x_k)^2} =   0
\end{equation}

\item[(ii)] The commutation relations of different pairs of growth are interpreted in terms of the Poisson bracket.

\begin{equation}
\left\{\mathcal{H}_j, \mathcal{H}_k\right\}=\frac{1}{f_{j k}^{2}}\left(\mathcal{H}_k-\mathcal{H}_j\right)
\end{equation}

thus the vector flows $X_{\mathcal{H}_j}$ induced by Hamiltonians $\mathcal{H}_j$ commute along the submanifolds $N_c$. 
$$N_c=\left\{(\boldsymbol{x}, \boldsymbol{p}): \mathcal{H}_j(\boldsymbol{x}, \boldsymbol{p})=c \text { for all } j\right\}$$ 
\end{itemize}
\end{thm}

This relationship is a classical analog of the relation between multiple radial $SLE(\kappa)$ and quantum Calogero-Moser system, first discovered in \cite{DC07}.

\newpage
\section{Commutation relations for multiple chordal SLE(0) systems } \label{Commutation relations and conformal invariance}

\subsection{Transformation of Loewner flow under coordinate change}
\label{transformation of Loewner under coordinate change}
\indent

In this section we show that the Loewner chain of a curve, when viewed in a different coordinate chart, is a time reparametrization of the Loewner chain in the standard coordinate chart but with different initial conditions. This result serves as a preliminary step towards understanding the local commutation relations and the conformal invariance of multiple SLE($\kappa$) systems.

Let us briefly review how Loewner chains transform under coordinate changes.

\begin{thm}[Loewner coordinate change in $\mathbb{H}$]
\label{Loewner coordinate change in H}
Let $\gamma=\gamma(t)$ be a continuous, non-crossing curve in $\overline{\mathbb{H}}$, and for simplicity we assume that $\gamma(0)=x \in \mathbb{R}$ and $\gamma(0, t] \subset \mathbb{H}$. Assume that $\gamma$ is generated by the Loewner chain

$$
\partial_t g_t(z)=\frac{2}{g_t(z)-W_t}, \dot{W}_t=b\left(W_t , g_t\left(z_1\right), \ldots, g_t\left(z_m\right)\right) \quad g_0(z)=z, W_0=x
$$

We only consider the case where $\dot{W}_t=b\left(W_t , g_t\left(z_1\right), \ldots, g_t\left(z_m\right)\right)$ for some $b: \mathbb{R} \times \mathbb{C}^m \rightarrow \mathbb{R}$. We allow $b$ to depend on the location of a collection of marked points in the flow $g_t$ but to keep the notation brief, we only write $\dot{W}_t=b\left(W_t\right)$ from now on. Under a conformal transformation $\tau: \mathcal{N} \rightarrow \mathbb{H}$, defined in a neighborhood $\mathcal{N}$ of $x$ such that $\gamma[0, T] \subset \mathcal{N}$ and such that $\Psi$ sends $\partial \mathcal{N} \cap \mathbb{R}$ to $\mathbb{R}$, the Loewner chain of the image curve $\tilde{\gamma}(t)=\Psi \circ \gamma(t)$ is as follows, at least for $0 \leq t \leq T$. Let $\tilde{g}_t$ denote the unique conformal transformation of $\mathbb{H} \backslash \tilde{\gamma}[0, t]$ onto $\mathbb{H}$ that satisfies the normalization $\tilde{g}_t(z)=z+o(1)$ as $z \rightarrow \infty$. Letting $\Psi_t=\tilde{g}_t \circ \tau \circ g_t^{-1}$, it can be computed that $\tilde{g}_t(z)$ evolves as

$$
\partial_t \tilde{g}_t(z ; x)=\frac{2 \Psi_t^{\prime}\left(W_t\right)^2}{\tilde{g}_t(z)-\tilde{W}_t}, \quad \tilde{g}_0(z)=z, \tilde{W}_0=\Psi(x)
$$

where $\tilde{W}_t=\tilde{g}_t \circ \Psi \circ \gamma(t)=\tilde{g}_t \circ \Psi \circ g_t^{-1}\left(W_t\right)=\Psi_t\left(W_t\right)$ is the driving function for the new flow. Note that $\tilde{W}_0=\Psi\left(W_0\right)=\Psi(x)$. The equation for $\partial_t \tilde{g}_t(z)$ shows that $\tilde{\gamma}$ is parameterized so that its half-plane capacity satisfies $\operatorname{hcap}(\tilde{\gamma}[0, t])=2 \sigma(t)$, where

$$
\sigma(t):=\int_0^t \Psi_s^{\prime}\left(W_s\right)^2 d s
$$

Furthermore, by $\widetilde{W}_t=\Psi_t\left(W_t\right)$ and $\dot{W}_t=b\left(W_t\right)$ we compute that when the time evolution of $W_t$ is sufficiently smooth (which is the only case that we consider) then the time evolution of $\widetilde{W}_t$ satisfies
\begin{equation}
\dot{\widetilde{W}}_t=\left(\partial_t \Psi_t\right)\left(W_t\right)+\Psi_t^{\prime}\left(W_t\right) \dot{W}_t=-3 \Psi_t^{\prime \prime}\left(W_t\right)+\Psi_t^{\prime}\left(W_t\right) v\left(W_t\right)
\end{equation}
The first part of the final equality comes from \cite{Lawler:book}, equation (4.35). 
\end{thm}

\begin{thm}\label{random coordinate change}
If the driving function $W_t$ in theorem (\ref{Loewner coordinate change in H}) is given by:
\begin{equation}
dW_t= \sqrt{\kappa}dB_t+ b(W_t;\Psi_t(W_1),\ldots,\Psi_t(W_n))
\end{equation}
Let $\widetilde{W}_t=\Psi_t\left(W_t\right)$, then 
\begin{equation}
\begin{aligned}
d \widetilde{W}_t= &\left(\partial_t \Psi_t\right)\left(W_t\right) d t+\Psi_t^{\prime}\left(W_t\right) d W_t+\frac{\kappa}{2} \Psi_t^{\prime \prime}\left(W_t\right) d t \\
=&\sqrt{\kappa}h_{t}'(W_t)dB_t+ \Psi_{t}'(W_t)b(W_t;\Psi_t(W_1),\ldots,\Psi_t(W_n))+\frac{\kappa-6}{2} \Psi_t^{\prime \prime}\left(W_t\right) d t 
\end{aligned}
\end{equation}

by changing time from $t$ to $s(t)=\int|\Psi_t^{\prime}\left(W_t\right)|^2 dt$ ,

\begin{equation}
d \widetilde{W}_s= \sqrt{\kappa}dB_s+ \Psi_{t(s)}^{\prime}(W_s)^{-1} b(W_{s};\Psi_{t(s)}(W_1),\ldots,\Psi_{t(s)}(W_n))ds+\frac{\kappa-6}{2} \frac{\Psi_{t(s)}^{\prime \prime}\left(W_{s}\right)}{\Psi_{t(s)}^{\prime }(W_s)^2} d s
\end{equation}
\end{thm}

\begin{remark}\label{drift term pre schwarz form}
By theorem (\ref{random coordinate change}), for conformal transformation $\tau$, the drift term in the marginal law is a pre-schwarz form, i.e.
$b=\tau^{\prime} \widetilde{b} \circ \tau+ \frac{6-\kappa}{2}\left(\log \tau^{\prime}\right)^{\prime}$.
\end{remark}

\begin{cor}\label{gamma 1 gamma2 capacity change}
Let $\gamma$, $\tilde{\gamma}$ be two hulls starting at $x \in \partial \mathbb{H}$ and $y \in \partial \mathbb{H}$ with capacity $\epsilon$ and $c \epsilon$ , let $g_{\epsilon}$ be the normalized map removing $\gamma$ and $\tilde{\epsilon}= \operatorname{hcap}(g_{\epsilon}\circ \gamma(t))$, then we have:

\begin{equation}
\tilde{\varepsilon}=
c \varepsilon\left(1-\frac{4\varepsilon}{(x-y)^2}\right)+o\left(\varepsilon^2\right)
\end{equation}

\end{cor}

\begin{proof}
From the Loewner equation, $\partial_t h_t^{\prime}(w)=- \frac{2h_t^{\prime}(w)}{\left(h_t(w)-x_t)\right)^2}$, which implies $h_{\varepsilon}^{\prime}(y)=1-\frac{4\varepsilon}{(x-y)^2}+o(\varepsilon)$. By conformal transformation $h_{\epsilon}(y)$, we get:
$$
\tilde{\varepsilon}=c\epsilon( h'_{\epsilon}(y)^2+ o(\epsilon)) =c \varepsilon\left(1-\frac{4\varepsilon}{(x-y)^2}\right)+o\left(\varepsilon^2\right)
$$

\end{proof}

\subsection{Local commutation relations and null vector equations in $\kappa=0$ case}
\label{commutation when kappa=0}

Next, in section \ref{commutation for kappa=0}, we discuss the theory for the multiple chordal SLE($0$) system, treating multiple chordal SLE($0$) curves as natural geometric objects without reference to multiple chordal SLE($\kappa$) systems, see \cite{ABKM 20}.
The defining properties of this ensemble of curves are geometric commutation and conformal invariance.

\begin{defn}
Let $\gamma_1,\dotsc,\gamma_{n}$ be simple disjoint smooth curves starting from $\{x_1,x_2,\ldots,x_n\}$ which are $n$ distinct points counterclockwise on the unit circle $\partial \mathbb{D}$. 
\begin{itemize}
\item[(i)] Each curve can be individually generated by a Loewner chain. The 
the Loewner equation for $g_t(z)$ is given by

\begin{equation}
\partial_t g_t(z)=\frac{2}{g_t(z)-x_j(t)}, \quad g_0(z)=z,   
\end{equation}

and the driving function $x_j(t)$ evolve as

$$
\left\{\begin{array}{l}
\mathrm{~d} x_{j}(t)=U_j\left(x_{1}(t),x_{2}(t),\ldots, x_{j}(t),\ldots,x_{n}(t)\right) \mathrm{d} t \\
\mathrm{~d} x_{k}(t)=\frac{x_{k}(t)-x_{j}(t)}{2} \mathrm{d} t,  k\neq j
\end{array}\right.
$$
where $U_j(\boldsymbol{x}): \mathfrak{X}^n \rightarrow \mathbb{R}$ is assumed to be smooth  
in the chamber $$\mathfrak{X}^n=\left\{(x_1,x_2,\ldots,x_n) \in \mathbb{R}^n \mid x_1<x_2<\ldots<x_n\right\}$$
\item[(ii)] The curves geometrically commute, meaning that the same collection of curves can be generated by applying the individual Loewner chains in any chosen order.
For example, we can first map out $\gamma_{\left[0, t_i\right]}^{(i)}$ using $h_{t_i}^{(i)}$, then mapping out $h_{t_i}^{(i)}\left(\gamma_{\left[0, t_j\right]}^{(j)}\right)$, or vice versa. The images are the same regardless of the order in which we map out the curves. 
\item[(iii)] Each curve $\gamma_j$ is Möbius invariant in $\mathbb{D}$. This means that if $\gamma_j$ is the curve generated by a Loewner flow and initial data $\boldsymbol{x}$, then its image $\phi\left(\gamma_j\right)$ under a conformal automorphism $\phi$ of $\mathbb{D}$ is, up to a time change, generated by the same flow with initial data $\phi(\boldsymbol{x})=\left(\phi\left(x_1\right), \ldots, \phi\left(x_n\right)\right)$.
Our definition for multiple chordal SLE(0) can be naturally extended to an arbitrary simply-connected domain $\Omega$ with a marked interior point $u$ via a conformal uniformizing map $\phi: \Omega \rightarrow \mathbb{D}$, sending $u$ to $0$.

\end{itemize}
\end{defn}

Under these dynamics, the driving function $x_j(t)$ evolves as specified by $U_j(\boldsymbol{x})$, while the points $x_k(t)$ simply follow the Loewner chain generated by $x_j(t)$. 
We define differential operator corresponding to $\gamma_j$ by
$$
\mathcal{M}_j=U_j(\boldsymbol{x}) \partial_j+\sum_{k \neq j} \frac{2}{x_k-x_j} \partial_k, \quad j=1, \ldots, n 
$$

A significant difference between $\kappa>0$ case and $\kappa=0$ case is that in the case of $\kappa=0$, the conditions $\partial_j U_k= \partial_{k} U_j$ can not be derived from the commutation relations which is equivalent
the existence of a function $\mathcal{U}(\boldsymbol{x})$ such that
$$
U_j=\partial_j \mathcal{U}
$$
Here $\mathcal{U}: \mathfrak{X}^n \rightarrow \mathbb{R}$ is smooth 
in the chamber $$\mathfrak{X}^n=\left\{(x_1,x_2,\ldots,x_n) \in \mathbb{R}^n \mid x_1<x_2<\ldots<x_n\right\}$$

We view multiple chordal SLE(0) system as the classical limit of multiple chordal SLE($\kappa$) system. For multiple chordal SLE($\kappa$) system, we have shown that the drift term $b_j(\boldsymbol{x})$ is given by 
$$b_j(\boldsymbol{x}) = \kappa \frac{\partial \log \mathcal{Z}(\boldsymbol{x})}{\partial x_j}$$
where $\mathcal{Z}(\boldsymbol{x})$ is a positive function satisfying the null vector equations (\ref{null vector equation for kappa 0}). The idea is that we expect that as $\kappa \rightarrow 0$, the limit $\lim_{\kappa\rightarrow 0}\mathcal{Z}(\boldsymbol{x})^{\kappa}= \mathcal{U}_j(\boldsymbol{x})$ exists (at least for some suitably chosen partition functions).

Therefore, we typically assume the existence of $\mathcal{U}(\boldsymbol{x})$ such that $U_j(\boldsymbol{x})= \partial_j \mathcal{U}(\boldsymbol{x})$
in the definition of the multiple chordal SLE(0) system.

\begin{thm}[See theorem (5.1) in \cite{ABKM20}] \label{commutation for kappa=0}
Let $\gamma_1, \ldots, \gamma_{ n}$ be smooth, simple curves in $\mathbb{H}$ that are generated by Loewner flows. If the curves locally geometrically commute then $\mathcal{L}_j$ satisfy the commutation relations

\begin{equation} \label{commutation relation for generators kappa=0}
\left[\mathcal{L}_j, \mathcal{L}_k\right]=\frac{4}{\left(x_k-x_j\right)^2}\left(\mathcal{L}_k-\mathcal{L}_j\right), \quad \text { for all } 1 \leq j \neq k \leq  n
\end{equation}

Moreover, under the additional assumption that $\partial_j U_k=\partial_k U_j$ for all $j, k$, the commutation relations hold for $\mathcal{L}_j$ iff there exists smooth functions $h_1, \ldots, h_{2 n}: \mathbb{R} \rightarrow \mathbb{R}$ such that
\begin{equation} \label{null vector equation for kappa 0}
\frac{1}{2} U_j^2+\sum_{k \neq j} \frac{2}{x_k-x_j} U_k-\sum_{k \neq j} \frac{6}{\left(x_k-x_j\right)^2}=0, \quad j=1, \ldots, n
\end{equation} 
\end{thm}

\section{Integrals of motion and the geometric characterization of Traces}

\subsection{Classical limit of multiple chordal SLE($\kappa$) system*} 
\label{classical limit of multiple chordal SLE system}
\indent

In this section, we construct multiple chordal SLE(0) systems by heuristically taking the classical limits of multiple chordal SLE($\kappa$) systems. Our construction of multiple chordal SLE(0) systems are self-consistent and does not depend on resolving the classical limit.

A key concept in this construction is the stationary relations that naturally emerge as a result of normalizing the partition functions.
For multiple chordal SLE($\kappa$) system, we have shown that the drift term $b_j(\boldsymbol{x})$ is given by 
$$b_j(\boldsymbol{x}) = \kappa \frac{\partial \log \mathcal{Z}(\boldsymbol{x})}{\partial x_j}$$
where $\mathcal{Z}(\boldsymbol{x})$ is a positive function satisfying the null vector equations (\ref{null vector equation for kappa 0}). 

As $\kappa \rightarrow 0$, we need to normalize the partition function. We expect that, for some suitably chosen partition functions, the limit $\mathcal{Z}(\boldsymbol{x})^{\kappa}$ exists as $\kappa \rightarrow 0$.

Recall that the chordal ground solution is given by the multiple contour integral:
\begin{equation}
\mathcal{J}_{\alpha}(\boldsymbol{x})=\oint_{\mathcal{C}_1} \ldots \oint_{\mathcal{C}_n} \Phi_\kappa(\boldsymbol{x}, \boldsymbol{\xi}) d \xi_m \ldots d \xi_1 .
\end{equation}
where $\Phi_\kappa(\boldsymbol{x}, \boldsymbol{\xi})$ is the multiple chordal SLE($\kappa$) master function, and  $\mathcal{C}_1,\mathcal{C}_2,\ldots,\mathcal{C}_m$ are non-intersecting Pochhammer contours. Pure partition functions $\mathcal{Z}_{\alpha}(\boldsymbol{x})$ are linear combinations of chordal ground solutions $\mathcal{J}_{\alpha}(\boldsymbol{x})$. Therefore, heuristically

\begin{equation}
\lim_{\kappa \rightarrow 0}\mathcal{Z}_\alpha(\boldsymbol{x})^{\kappa}= \lim_{\kappa \rightarrow 0}\left(\oint_{\mathcal{C}_1} \ldots \oint_{\mathcal{C}_m} \Phi(\boldsymbol{x}, \boldsymbol{\xi})^{\frac{1}{\kappa}}d\boldsymbol{\xi}\right)^{\kappa}
\end{equation}
where $\Phi(\boldsymbol{x},\boldsymbol{\xi})$ is the multiple chordal SLE(0) master function, see definition (\ref{multiple chordal SLE(0) master function}).
By the steepest decent method,
the contour integral $\oint_{\mathcal{C}_1} \ldots \oint_{\mathcal{C}_m} \Phi(\boldsymbol{x}, \boldsymbol{\xi})^{\frac{1}{\kappa}}d\boldsymbol{\xi}$ is approximated by the value of $\Phi(\boldsymbol{x}, \boldsymbol{\xi})$ at the stationary phase, i.e., the critical points of $\Phi(\boldsymbol{x}, \boldsymbol{\xi})$.

\begin{conjecture} We conjecture that for pure partition function $\mathcal{Z}_{\alpha}(\boldsymbol{z},u)$ and $\mathcal{Z}_{\alpha}(\boldsymbol{x})$
\begin{itemize}
    \item As $\kappa \rightarrow 0$, the limit $\underset{\kappa \rightarrow 0}{\lim}\mathcal{Z}_{\alpha}(\boldsymbol{z},u)^{\kappa}$ exist and concentrate on critical points of the master function. i.e.
\begin{equation}
\underset{\kappa \rightarrow 0}{\lim}\mathcal{Z}_{\alpha}(\boldsymbol{z})^{\kappa}=  \Phi(\boldsymbol{x}, \boldsymbol{\xi},u)     
\end{equation}   
 where $\boldsymbol{\xi}$ is a critical point of the multiple chordal SLE(0) master function $\Phi(\boldsymbol{x}, \boldsymbol{\xi},u)$. 

\item In the special case $u=\infty$, as $\kappa \rightarrow 0$, the limit  $\underset{\kappa \rightarrow 0}{\lim}\mathcal{Z}_{\alpha}^{\kappa}(\boldsymbol{x}) $ exist and concentrate on critical points of the master function. i.e.

 \begin{equation}
\underset{\kappa \rightarrow 0}{\lim}\mathcal{Z}_{\alpha}(\boldsymbol{x})^{\kappa}=  \Phi(\boldsymbol{x}, \boldsymbol{\xi})     
\end{equation}   
 where $\boldsymbol{\xi}$ is a critical point of the multiple chordal SLE(0) master function $\Phi(\boldsymbol{x}, \boldsymbol{\xi})$. 
\end{itemize}
\end{conjecture}

The \textbf{stationary relations} correspond to the critical points of $\Phi(\boldsymbol{z}, \boldsymbol{\xi})$, known as the multiple chordal SLE(0) master functions, or the master function for the Knizhnik-Zamolodchikov (KZ) equations.

\begin{defn}[Multiple chordal SLE(0) master function]\label{multiple chordal SLE(0) master function}
For growth points $\boldsymbol{z}=\{z_1,z_2,\ldots,z_n\}$, a marked point $u$ and force points $\boldsymbol{\xi}=\{\xi_1,\xi_2,\ldots,\xi_m\}$. 
We define the Coulomb gas partition function by
\begin{equation}
\begin{aligned}
 \Phi(\boldsymbol{z}, \boldsymbol{\xi})= 
 & 
 \prod_{1 \leq i<j \leq n}(z_i-z_j)^{2} \prod_{1 \leq i<j \leq m}(\xi_i-\xi_j)^{8} \prod_{i=1}^{n} \prod_{j=1}^m\left(z_i-\xi_j\right)^{-4} \\
 &
 \prod_{1 \leq i \leq n}
 (z_{i}-u)^{2(-2-n+2m)}\prod_{1 \leq j \leq n}(\xi_{j}-u)^{4(2+n-2m)}    
\end{aligned}
\end{equation}

when $u=\infty$, the master function is simplified as 
\begin{equation}
    \Phi(\boldsymbol{z}, \boldsymbol{\xi})= \prod_{1 \leq i<j \leq n}(z_i-z_j)^{2} \prod_{1 \leq i<j \leq m}(\xi_i-\xi_j)^{8} \prod_{i=1}^{n} \prod_{j=1}^m\left(z_i-\xi_j\right)^{-4}
\end{equation}

\end{defn}

\begin{thm}
    The stationary relation between $\boldsymbol{x}$ and $\boldsymbol{\xi}$ is equivalent to $\boldsymbol{\xi}$ is the critical point of the master function  $\Phi(\boldsymbol{z}, \boldsymbol{\xi})$.
\end{thm}

\begin{proof} By direct computing the partial derivative with respect to $\xi_k$, we obtain that:
    \begin{equation}
        \frac{\partial \Phi(\boldsymbol{z}, \boldsymbol{\xi})}{\partial \xi_k}= 2\left(-\sum_{j=1}^{n} \frac{2}{\xi_k-z_j}+\sum_{l \neq k} \frac{4}{\xi_k-\xi_l}+ \frac{2n-4m+4}{\xi_k-u}\right)\Phi(\boldsymbol{z}, \boldsymbol{\xi})
    \end{equation}
    Thus the stationary relations are equivalent to 
    \begin{equation}
    \frac{\partial \Phi(\boldsymbol{z}, \boldsymbol{\xi})}{\partial \xi_k}=0, \quad k=1,2,\ldots,n
    \end{equation}
\end{proof}

\begin{thm}[Conformal invariance of stationary relations] 
The stationary relations are preserved under the M\"obius transformation ${\rm Aut}(\mathbb{H})$.
\end{thm}
\begin{proof}
The conformal invariance of stationary relations naturally follows from the conformal invariance of the Coulomb gas correlation.
\end{proof}

\subsection{Stationary relations imply commutation relations in $\kappa=0$ case} 
\label{Stationary relations imply commutation}
\indent

In definition (\ref{multiple chordal SLE(0) Loewner chain via stationary relations}), we define the multiple radial SLE(0) through stationary relations. In this section , we show that the stationary relations between \( \boldsymbol{\xi} \) and \( \boldsymbol{x} \) naturally imply the existence of partition functions that satisfy the null vector equations for \( \kappa = 0 \). Thus, our special construction of multiple radial SLE(0) systems fits into the general framework of multiple radial SLE(0) systems.

\begin{thm}[See theorem (2.8) in \cite{ABKM20}]
\label{thm stationary relations imply commutation} 
For distinct boundary points $\boldsymbol{x}$ and $\boldsymbol{\xi}$ smoothly depending on $\boldsymbol{x}$ by solving the stationary relations, we define

\begin{equation}
\mathcal{Z}(\boldsymbol{x}):=\prod_{1 \leq j<k \leq  n}\left(x_j-x_k\right)^2 \prod_{1 \leq l<m \leq n}\left(\xi_{l}(\boldsymbol{x})-\xi_{m}(\boldsymbol{x})\right)^8 \prod_{k=1}^{ n} \prod_{l=1}^n\left(x_k-\xi_l(\boldsymbol{x})\right)^{-4}    
\end{equation}

Then $\mathcal{Z}(\boldsymbol{x})$ is strictly positive, and is a $(3,0)$-differential at each coordinate of $\boldsymbol{x}$ with respect to ${\rm Aut}(\mathbb{H},\infty)$. Moreover

\begin{equation} \label{Uj in x xi}
 U(\boldsymbol{x})=\partial_{x_j} \log \mathcal{Z}(\boldsymbol{x})=\sum_{k \neq j} \frac{2}{x_j-x_k}-\sum_{k=1}^n \frac{4}{x_j-\xi_k(\boldsymbol{x})}, \quad j=1, \ldots, n   
\end{equation}

are real-valued, translation invariant, homogeneous of degree -1 , and satisfy the system of null vector equations (\ref{commutation for kappa=0}).
\end{thm}

\begin{thm}[Ward's identities]
\label{ward identities kappa 0}
For distinct real points $\boldsymbol{x}$, the functions $U_{j}$  satisfy the conformal Ward identities

$$
\sum_{j=1}^{n} U_{j}=0, \quad \sum_{j=1}^{n} x_j U_{ j}=(n-2m)^2-n-4m
$$
\end{thm}

\begin{proof}[Proof of Theorem (\ref{thm stationary relations imply commutation})]
Please refer to the proof of Theorem (2.8) in \cite{ABKM20}, we attach is here for completeness of the paper.

We first prove that the $U_{j}$ of (\ref{Uj in x xi}) satisfy the null vector equations. Write $U_{ j}=U_j$ throughout. Real-valuedness of $U_j$ follows from $x_k \in \mathbb{R}$ and the $\xi_k$ occurring in complex conjugate pairs. For the translation invariance, note that the pole functions $\xi_{ k}$ clearly satisfy $\xi_{ k}(\boldsymbol{x}+h)=\xi_{ k}(\boldsymbol{x})+h$ for $h \in \mathbb{R}$, since if $R \in \mathcal{CR}_{n.m}(\boldsymbol{x})$ then $R(z-h) \in$ $\mathcal{CR}_{n,m}(\boldsymbol{x}+h)$. Similarly, $r R\left(r^{-1} z\right) \in \mathcal{CR}_{n,m}(r \boldsymbol{x})$ for $r>0$ , hence $\xi_{ k}(r \boldsymbol{x})=r \xi_{ k}(\boldsymbol{x})$. The homogeneity of $U_{ j}$ follows from these properties of the pole function and the definition (\ref{Uj in x xi}).

For the proof that the $U_j$ satisfy the null vector equations it is convenient to introduce

$$
u_j:=U_j-\sum_{k \neq j} \frac{2}{x_j-x_k} .
$$

All we need to check is that $u_j$ 's satisfy the algebraic equations

$$
\frac{1}{2} u_j^2+2 \sum_{k \neq j} \frac{u_j-u_k}{x_j-x_k}=0, \quad j=1, \ldots, n
$$

Using the stationary relation (\ref{Stationary relations}), we have

$$
\begin{aligned}
\frac{1}{8}\left(\frac{1}{2} u_j^2\right. & \left.+2 \sum_{k \neq j} \frac{u_j-u_k}{x_j-x_k}\right)=\left(\sum_{l=1}^n \frac{1}{\xi_l-x_j}\right)^2+\sum_{k \neq j} \sum_{l=1}^n \frac{1}{\left(\xi_l-x_j\right)\left(\xi_l-x_k\right)} \\
& =\sum_{l=1}^n \frac{1}{\xi_l-x_j} \sum_{k=1}^{2 n} \frac{1}{\xi_l-x_k}+2 \sum_{k<l} \frac{1}{\left(\xi_k-x_j\right)\left(\xi_l-x_j\right)} \\
& =\sum_{l=1}^n \frac{1}{\xi_l-x_j} \sum_{k \neq l} \frac{2}{\xi_l-\xi_k}+2 \sum_{k<l} \frac{1}{\left(\xi_k-x_j\right)\left(\xi_l-x_j\right)}=0
\end{aligned}
$$

Here we used the identity

$$
\sum_{k<l} \frac{1}{\left(\xi_k-x_j\right)\left(\xi_l-x_j\right)}=\sum_{k<l}-\frac{1}{\left(\xi_k-x_j\right)\left(\xi_k-\xi_l\right)}+\frac{1}{\left(\xi_l-x_j\right)\left(\xi_k-\xi_l\right)}
$$

\end{proof}

\begin{proof}[Proof of Theorem (\ref{ward identities kappa 0})]
    For convenience we drop the dependence on $\alpha$ throughout. The first conformal Ward identity holds by symmetry:

$$
\sum_{j=1}^{ n} U_j=\sum_{j=1}^{ n} \sum_{k \neq j} \frac{2}{x_j-x_k}+\sum_{j=1}^{ n} \sum_{k=1}^m \frac{4}{\xi_k-x_j}=\sum_{k=1}^m \sum_{l \neq k} \frac{8}{\xi_k-\xi_l}=0 .
$$

To find a sufficient and necessary condition for the second identity $\sum_{j=1}^{n } x_j U_j=(n-2m)^2-n-4m$, we compute

$$
\sum_{j=1}^{ n} \sum_{k \neq j} \frac{x_j}{x_j-x_k}=\frac{n(n-1)}{2}, \quad \sum_{j=1}^{ n} \sum_{k=1}^m \frac{x_j}{\xi_k-x_j}=-mn+\sum_{k=1}^, \sum_{l \neq k} \frac{2 \xi_k}{\xi_k-\xi_l}=-mn+m(m-1)
$$

+
Thus we find $\sum_{j=1}^{ n} x_j U_j= n^2- n-4mn+4 m^2-4m= (n-2m)^2-n-4m$, as claimed. 
\end{proof} 
\subsection{Real rational function with prescribed critical points}
\indent

In this section, we analyze the properties of rational functions $R(z)\in \mathcal{R}_{m,n}(\boldsymbol{z})$.

\begin{proof}[Proof of Theorem (\ref{stationary relations rational functions})]
  Please refer to Lemma 4.1 in \cite{ABKM20}. we summarize the proof here for the completeness of the paper.

 Fix $n \geq 1$ and let $\boldsymbol{x}=\left\{x_1, \ldots, x_{n}\right\}$ be real and distinct. The following are true:
\begin{itemize}
    \item[(i)] For $0 \leq m \leq n$ consider the meromorphic function

\begin{equation} \label{prod f}
f(z)=\frac{\prod_{j=1}^{ n}\left(z-x_j\right)}{\prod_{k=1}^d\left(z-\xi_k\right)^2}
\end{equation}

If $\xi_1, \ldots, \xi_m \in \mathbb{C}$ are all distinct then $f(z)$ has a partial fraction expansion

\begin{equation}
\label{meromorphic f}
f(z)=p(z)+\sum_{k=1}^d \frac{A_k}{\left(z-\xi_k\right)^2}+\sum_{k=1}^d \frac{B_k}{z-\xi_k}
\end{equation}

where $p(z)$ is a polynomial of degree $\max \{2 n-2 m, 0\}$, and the constants $A_k$ and $B_k$ are

\begin{equation}
A_k=\frac{\prod_{j=1}^{n}\left(\xi_k-x_j\right)}{\prod_{l \neq k}\left(\xi_k-\xi_l\right)^2}, \quad B_k=\left(\sum_{j=1}^{ n} \frac{1}{\xi_k-x_j}-\sum_{l \neq k} \frac{2}{\xi_k-\xi_l}\right) A_k, \quad k=1, \ldots, d
\end{equation}
\item[(ii)] If $\xi_1, \ldots, \xi_m \in \mathbb{C}$ are all distinct then the meromorphic function (\ref{meromorphic f}) has a rational primitive iff all $B_k=0$.
\item[(iii)] Up to a real multiplicative constant, the derivative of any $R \in \mathcal{CR}_{n,m}(\boldsymbol{x})$ factorizes as

\begin{equation}\label{R' form}
R^{\prime}(z)=\frac{\prod_{j=1}^{n}\left(z-x_j\right)}{\prod_{k=1}^{m}\left(z-\xi_k\right)^2}  
\end{equation}

where $\left\{\xi_1, \ldots, \xi_{m}\right\}$ are the poles of $R$, considered as a multi-set in the case of multiplicities.

\end{itemize}

In part (i) it is clear that (\ref{meromorphic f}) follows from (\ref{prod f}) and each pole being of order at most two, the latter of which follows from the assumption that the $\zeta_k$ are distinct. The formula for $A_k$ follows from (\ref{prod f}) and (\ref{meromorphic f}) by

$$
A_k=\lim _{z \rightarrow \zeta_k} f(z)\left(z-\zeta_k\right)^2=\frac{\prod_{j=1}^{2 n}\left(\zeta_k-x_j\right)}{\prod_{l \neq k}\left(\zeta_k-\zeta_l\right)^2}
$$

The formula for $B_k$ follows from

$$
B_k=\lim _{z \rightarrow \zeta_k}\left(f(z)\left(z-\zeta_k\right)^2\right)^{\prime}=\lim _{z \rightarrow \zeta_k}\left(\sum_{j=1}^{2 n} \frac{1}{z-x_j}-\sum_{l \neq k} \frac{2}{z-\zeta_l}\right) f(z)\left(z-\zeta_k\right)^2
$$

Part (ii) is obvious. 

In part (iii) the factorization of $R^{\prime}$ into the form (\ref{R' form}) follows from a combination of $$R^{\prime}=\frac{\left(P^{\prime} Q-Q P^{\prime}\right)}{Q^2}$$ the critical points of $R$ being precisely the points in $\boldsymbol{x}$, and the finite poles of $R$ being the zeros of $Q$. The real-valuedness of the multiplicative constant follows
from $P$ and $Q$ having real coefficients.

\end{proof}
\subsection{Field integral of motions and real locus as flow lines}
\indent

In the chordal case, similarly, we can construct a field integral of motions for the multiple chordal Loewner flows.

\begin{lemma} Let $z_1,z_2,\ldots,z_n$ be distinct points on the unit circle, for each $z \in \overline{\mathbb{H}}$

\begin{equation}
    N_t(z)=(g_t(z)-u)^{2m-n-2}g'_{t}(z)\frac{\prod_{k=1}^{n}(g_t(z)-z_k(t))}{\prod_{j=1}^{m}(g_t(z)-\xi_j(t))^2}
\end{equation}

is an integral of motion on the interval $[0,\tau_t \wedge \tau)$ for the multiple chordal Loewner flow with parametrization $\nu_j(t)=1$, $\nu_k(t)=0, k\neq j$, i.e. $j$-th curve is growing.
\end{lemma}

\begin{proof}
By the Loewner equation, the following identities hold:
\begin{equation}
\left\{
\begin{aligned}
 &\frac{dz_j(t)}{dt} =\sum_{k\neq j}\frac{2}{z_{k}(t)-z_j(t)}- \sum_{l}\frac{4}{z_j(t)-\xi_{l}(t)} \\
 &\frac{dz_{k}(t)}{dt}=\frac{2}{z_k(t)-z_{j}(t)},  k \neq j \\
  &  \frac{d\xi_{l}(t)}{dt}= \frac{2}{\xi_{l}(t)-z_j(t)} \\
  &  \frac{d u(t)}{dt}= \frac{2}{u(t)-z_j(t)} \\
   &  \frac{dg_t(z)}{dt} =  \frac{2}{g_t(z)-z_j(t)} \\       
\end{aligned}
\right.
\end{equation}

 \begin{equation}\label{derivatives of terms}
\left\{
\begin{aligned}
& \frac{d\log g_{t}'(z)}{dt}=- \frac{2}{(z_j(t)-g_t)^2} \\
& \frac{d\log(z_k(t) - g_t(z))}{dt}=\frac{1}{z_k-g_t}\left(\frac{2}{z_k(t)-z_{j}(t)}-\frac{2}{g_t(z)-z_j(t)}\right)=-\frac{2}{(z_k(t)-z_j(t))(g_t(z)-z_j(t))},   k\neq j \\
&\frac{d\log(\xi_l(t) - g_t(z))}{dt}=
\frac{1}{\xi_l-g_t}\left(\frac{2}{\xi_{l}(t)-z_j(t)}- \frac{2}{g_t(z)-z_j(t)}\right)
= -\frac{2}{(\xi_l(t)-z_j(t))(g_t(z)-z_j(t))}   \\
& \frac{d\log(u - g_t(z))}{dt}=
\frac{1}{u-g_t}\left(\frac{2}{u-z_j(t)}- \frac{2}{g_t(z)-z_j(t)}\right)
= -\frac{2}{(u-z_j(t))(g_t(z)-z_j(t))}\\
&\frac{d\log(z_j(t)-g_t(z))}{dt}=\frac{1}{z_j(t)-g_t}\left(\sum_{k\neq j}\frac{2}{z_j(t)-z_k(t)}-\sum_{l}\frac{4}{z_j(t)-\xi_{l}(t)}+\frac{4m-2n-4}{z_j-u} -\frac{2}{g_t-z_j(t)}\right)
\end{aligned}
\right.
 \end{equation} 

\begin{equation} \label{Njt}
\log N_{t}(z)=(2m-n-2)\log(g_t(z)-u)+\log(g'_t(z))+ \sum_{k=1}^{n}\log(g_t(z)-z_k(t))-2\sum_{k=1}^{n}\log(g_t(z)-\xi_j(t))    
\end{equation}
By substituting equations (\ref{derivatives of terms}) into equation (\ref{Njt}).
\begin{equation} \label{dNj t}
\frac{d\log N_t(z)}{dt} =
 \underbrace{\begin{aligned}
 &-(2m-n-2)\frac{2}{(u-z_j(t))(g_t(z)-z_j(t))}-\frac{2}{(z_j(t)-g_t)^2}\\&
 -\sum_{k=1}^{n}\frac{2}{(z_k(t)-z_j(t))(g_t(z)-z_j(t))} +2\sum_{l=1}^{m}\frac{2}{(\xi_l(t)-z_j(t))(g_t(z)-z_j(t))} \\
\end{aligned}}_{N^{j}_{t}(z)}
\end{equation}

We denote the sum on the right hand side of (\ref{dNj t}) by $N^{j}_{t}(z)$.
By direct computations we obtain that all terms canceled out,
$$\frac{ d\log N_t(z)}{dt} = N_{t}^{j}(t) =0$$

\end{proof}

\begin{proof}[Proof of Theorem (\ref{integral of motion in H})]

Note that for $\nu(t)$ parametrization
$$
\partial_t g_t(z)=\sum_{j=1}^n \nu_j(t)\frac{2}{g_t(z)-z_j(t)}, \quad g_0(z)=z,
$$

  \begin{equation}
\left\{
\begin{aligned}
& \frac{d\log g_{t}'(z)}{dt}=-\sum_j \frac{2}{(z_j(t)-g_t)^2} \\
&\frac{d\log(\xi_l(t) - g_t(z))}{dt}=\sum_{j=1}^{n}
\frac{\nu_j(t)}{\xi_l-g_t}\left(\frac{2}{\xi_{l}(t)-z_j(t)}- \frac{2}{g_t(z)-z_j(t)}\right)
= -\sum_{j=1}^{n}\frac{2}{(\xi_l(t)-z_j(t))(g_t(z)-z_j(t))}   \\
& \frac{d\log(u - g_t(z))}{dt}=\sum_{j=1}^{n}
\frac{\nu_j(t)}{u-g_t}\left(\frac{2}{u-z_j(t)}- \frac{2}{g_t(z)-z_j(t)}\right)
= -\sum_{j=1}^{n}\frac{2\nu_j(t)}{(u-z_j(t))(g_t(z)-z_j(t))}\\
&\frac{d\log(z_j(t)-g_t(z))}{dt}=\frac{\nu_j(t)
}{z_j(t)-g_t}\left(\sum_{k\neq j}\frac{2}{z_j(t)-z_k(t)}-\sum_{l}\frac{4}{z_j(t)-\xi_{l}(t)}+\frac{4m-2n-4}{z_j-u} -\frac{2}{g_t-z_j(t)}\right)\\
&-\sum_{k\neq j}\nu_{k}(t)\frac{2}{(z_k(t)-z_j(t))(g_t(z)-z_j(t))}
\end{aligned}
\right.
 \end{equation} 
 By plugging in these identities, we obtain that 

$$\frac{d \log N_t(z)}{dt}= \sum_{j=1}^{n}\nu_j(t) N^{j}_t(z)=0$$

\end{proof}

\subsection{Classical limit of martingale observables* } \label{Multiple Chordal Martingale Observable}
\indent

In this section, we discuss how the field integral of motion is heuristically derived as the classical limit of martingale observable constructed in conformal field theory.

Based on the SLE-CFT correspondence, the multiple chordal SLE($\kappa$) system can be coupled to a conformal field theory. We will construct this conformal field theory using vertex operators, based on \cite{KM13,KM21}.

\begin{defn}[$n$-leg operator with screening charges]
    Consider the following charge distribution on the Riemann sphere.

$$
\boldsymbol{\beta}=2b\delta_{\infty}
$$

$$\boldsymbol{\tau_1}=\sum_{j=1}^{n} a \delta_{z_j}-\sum_{k=1}^m 2 a \delta_{\xi_k}-(n-2m)a \delta_{\infty}$$

$$\boldsymbol{\tau_2}=-\sigma \delta_{\infty}+\sigma \delta_z $$

where the parameter $\sigma= \frac{1}{a}$.

The $n$-leg operator with screening charges $\boldsymbol{\xi}$ and background charge $\boldsymbol{\beta}$ is given by the OPE exponential: 

\begin{equation}
\mathcal{O}_{\boldsymbol{\beta}}[\boldsymbol{\tau_1}]=\frac{C_{(b)}[\boldsymbol{\tau_1}+\boldsymbol{\beta}]}{C_{(b)}[\boldsymbol{\beta}]} \mathrm{e}^{\odot i \Phi[\boldsymbol{\tau_1}]}
\end{equation}

\end{defn}

\begin{defn}[Screening fields]

For each link pattern $\alpha$, we can choose closed contours $\mathcal{C}_1, \ldots, \mathcal{C}_n$ along which we may integrate the $\boldsymbol{\xi}$ variables to screen the vertex fields.
Let $\mathcal{S}$ be the screening operator, we define the screening operation as 
$$
\mathcal{S}\mathcal{O}_{\boldsymbol{\beta}}[\boldsymbol{\tau_1}]=\oint_{\mathcal{C}_1} \ldots \oint_{\mathcal{C}_n} \mathcal{O}_{\boldsymbol{\beta}}[\boldsymbol{\tau_1}]
$$

Meanwhile, we integrate the correlation function $\mathbf{E}\mathcal{O}_{\boldsymbol{\beta}}[\boldsymbol{\tau_1}]=\Phi_\kappa(\boldsymbol{z}, \boldsymbol{\xi})$ , the conformal dimension is 1 at the $\boldsymbol{\xi}$ points, i.e. since $\lambda_b(-2 a)=1$. This leads to the partition function for the corresponding multiple chordal SLE($\kappa$) system:
$$
\mathcal{Z}_\kappa(\boldsymbol{z}):=\mathbf{E}\mathcal{S}\mathcal{O}_{\boldsymbol{\beta}}[\boldsymbol{\tau_1}]=\oint_{\mathcal{C}_1} \ldots \oint_{\mathcal{C}_n} \Phi_\kappa(\boldsymbol{z}, \boldsymbol{\xi}) d \xi_n \ldots d \xi_1 .
$$

\end{defn}

\begin{thm}[Martingale observable]
 For any tensor product $X$ of fields in the OPE family $\mathcal{F}_{\boldsymbol{\beta}}$ of $\Phi_{\boldsymbol{\beta}}$,
\begin{equation}
M_t(X)=\frac{\mathbf{E}\mathcal{S} \mathcal{O}_{\boldsymbol{\beta}}[\boldsymbol{\tau_1}] X}{\mathbf{E}\mathcal{S} \mathcal{O}_{\boldsymbol{\beta}}[\boldsymbol{\tau_1}]} \| g_t^{-1}
\end{equation}
is a local martingale, where $g_t(z)$ is the Loewner map for multiple chordal SLE($\kappa$) system associated to $\mathcal{Z}_\kappa(\boldsymbol{z})=\mathbf{E}\mathcal{S} \mathcal{O}_{\boldsymbol{\beta}}[\boldsymbol{\tau_1}]$

\end{thm}

\begin{remark}
The Martingale observable theorem can be generalized to the linear combination of screening fields $\mathcal{S}_{\alpha}\mathcal{O}_{\boldsymbol{\beta}}[\boldsymbol{\tau}_1]$ with respect to different integral contours, where
 $Y= \sum_{\alpha} \sigma_{\alpha}\mathcal{S}_{\alpha} \mathcal{O}_{\boldsymbol{\beta}}$, where $\sigma_{\alpha} \in \mathbb{R}$.

\end{remark}

\begin{cor}
Let the divisor $ \boldsymbol{\tau_2}=-\frac{\sigma}{2} \delta_{0}-\frac{\sigma}{2}\ \delta_{\infty}+\sigma \delta_z$ where the parameter $\sigma= \frac{1}{a}$, and insert $ X = \mathcal{O}_{\boldsymbol{\beta}}[\boldsymbol{\tau_2}]$.

\begin{equation}
M_{t,\kappa}(z)=\frac{\mathbf{E}\mathcal{S} \mathcal{O}_{\boldsymbol{\beta}}[\boldsymbol{\tau_1}]\mathcal{O}_{\boldsymbol{\beta}}[\boldsymbol{\tau_2}]}
{\mathbf{E}\mathcal{S} \mathcal{O}_{\boldsymbol{\beta}}[\boldsymbol{\tau_1}]} \| g_t^{-1}
\end{equation}
is local martingale where $g_t(z)$ is the Loewner map for multiple chordal SLE($\kappa$) system associated to $\mathcal{Z}_\kappa(\boldsymbol{z})=\mathbf{E}\mathcal{S} \mathcal{O}_{\boldsymbol{\beta}}[\boldsymbol{\tau_1}]$
.
\end{cor}

Explicit computation shows that:

$$
\begin{aligned}
    &\mathbf{E}\oint_{\mathcal{C}_1} \ldots \oint_{\mathcal{C}_n} \mathcal{O}_{\boldsymbol{\beta}}[\boldsymbol{\tau_1}]\mathcal{O}_{\boldsymbol{\beta}}[\boldsymbol{\tau_2}]\\
    &=\oint_{\mathcal{C}_1} \ldots \oint_{\mathcal{C}_n} \prod_{1 \leq i<j \leq n}(z_i-z_j)^{a^2} \prod_{1 \leq i<j \leq m}(\xi_i-\xi_j)^{4 a^2} \prod_{i=1}^{n} \prod_{j=1}^m\left(z_i-\xi_j\right)^{-2 a^2} \\
& g^{\prime}(z_j) ^{\lambda_b(a)}g^{\prime}(z) ^{\lambda_b(\sigma)} (z-z_j)^{\sigma a}(z-\xi_k)^{-2\sigma a}
\end{aligned}
$$
$$
\begin{aligned}
    &\mathbf{E}\oint_{\mathcal{C}_1} \ldots \oint_{\mathcal{C}_n} \mathcal{O}_{\boldsymbol{\beta}}[\boldsymbol{\tau_1}] \\
    &=\oint_{\mathcal{C}_1} \ldots \oint_{\mathcal{C}_n} \prod_{1 \leq i<j \leq n}(z_i-z_j)^{a^2} \prod_{1 \leq i<j \leq m}(\xi_i-\xi_j)^{4 a^2} \prod_{i=1}^{n} \prod_{j=1}^m\left(z_i-\xi_j\right)^{-2 a^2} 
g^{\prime}(z_j) ^{\lambda_b(a)}
\end{aligned}
$$

Let the parameter $\sigma= \frac{1}{a}$. Then
as $\kappa \rightarrow 0$, the multiple integrals concentrate on the poles.

\begin{equation}
\begin{aligned}
 &M_{t,0}(z)=\lim_{\kappa \rightarrow 0}\frac{\mathbf{E}\oint_{\mathcal{C}_1} \ldots \oint_{\mathcal{C}_n} \mathcal{O}_{\boldsymbol{\beta}}[\boldsymbol{\tau_1}]\mathcal{O}_{\boldsymbol{\beta}}[\boldsymbol{\tau_2}]}
{\mathbf{E}\oint_{\mathcal{C}_1} \ldots \oint_{\mathcal{C}_n} \mathcal{O}_{\boldsymbol{\beta}}[\boldsymbol{\tau_1}]}\\
&=g_{t}'(z)\frac{\prod_{k=1}^{n}(z-z_k)}{\prod_{j=1}^{m}(z-\xi_j)^2}
\end{aligned}
\end{equation}
which is exactly the integral of motion we use.

$M_{t,\kappa}(z)$ is a $(\lambda_b(\sigma),0)$ differential with respect to $z$, where $\lambda_b(\sigma)= \frac{1}{2a^2}-\frac{b}{a}$.
By taking the limit $\kappa \rightarrow 0$,
$\lim_{\kappa \rightarrow 0}\lambda_b(\sigma)=1$, thus $M_{t,0}(z)$ is a (1,0) differential.

\begin{remark}
   The integral of motion $N_{t}(z)$ can be verified through direct computation.
\end{remark}

\subsection{Enumeration and topological link pattern}
\label{Enumerative geometry of chordal multiple SLE(0)}
\indent

In this section, we present theorems concerning the classification of rational functions \( R(z) \in \mathcal{R}(\boldsymbol{z}) \), which correspond bijectively to the critical points of the master functions of rational Knizhnik–Zamolodchikov (KZ) equations and, equivalently, to the multiple chordal $\mathrm{SLE}(0)$ systems. These results were established in \cite{S02a, SV03}.

\begin{thm}
 For generic $\boldsymbol{z}$ on the real line and $m$ poles $\boldsymbol{\xi}$.
\begin{itemize}

\item \rm(Underscreening) If $n+1-m>m$, then for generic $\boldsymbol{z}$ the real locus $\Gamma(R)$ forms a $(n,m)$-link,  each link pattern can be realized by a unique $R\in \mathcal{CR}_{m,n}(\boldsymbol{z})$  (up a real multiplicative constant).
\item \rm(Threshold) If $n+1-m=m$, there are no such $R \in \mathcal{CR}_{m,n}(\boldsymbol{z})$.
\item \rm(Overscreening) If $0\leq n+1-m<m$, for generic $\boldsymbol{z}$, the real locus $\Gamma(R)$ forms a $(n,m)$-link, each link pattern can be realized by a continuous family of $R\in \mathcal{CR}_{m,n}(\boldsymbol{z})$.
\item \rm(Upperbound) If $m >n$, there are no such $R\in \mathcal{CR}_{m,n}(\boldsymbol{z})$.
\end{itemize}

\end{thm}

\begin{proof}
By the degenerate case of Theorem 2 in \cite{EG02}, each link pattern can be realized by a rational function $\mathcal{CR}_{m,n}(\boldsymbol{z})$. The uniqueness follows from the following theorem (\ref{enumeration critical points}).
\end{proof}

We can state the classification theorems in terms of critical points of rational KZ equations.

\begin{thm}[Theorem 1 in \cite{SV03}]
\label{enumeration critical points}
 Let $n \in \mathbb{Z}_{>0}$, $m \in \mathbb{Z}_{\geq 0}$, for generic $\boldsymbol{z}$,
 \begin{itemize}
\item If $n+1-m>m$, then $\Phi_{m,n}(\boldsymbol{z}, \boldsymbol{\xi})$  has exactly $\#LP(n,m)$ critical points.
\item If $l(n)+1-m=m$,  $\Phi_{m,n}(\boldsymbol{z}, \boldsymbol{\xi})$ does not have critical points.
\item If $0 \leq n+1-m<m$, then for generic $z$,  $\Phi_{m,n}(\boldsymbol{z}, \boldsymbol{\xi})$ has non-isolated critical points. Written in symmetric coordinates $\lambda_1, \ldots, \lambda_k$, the critical set consists of $\#LP(n, n+1-m)$ straight lines in the space $\mathbb{C}_\lambda^m$.
\item If $n+1-m<0$, $\Phi_{m,n}(\boldsymbol{z}, \boldsymbol{\xi})$ does not have critical points.
\end{itemize}
\end{thm}

\subsection{Examples: underscreening}
\label{underscreening}
\indent

In this section, we present examples of the multiple chordal SLE(0) systems in the $\mathbb{D}$ uniformization.

Consider the holomorphic map $\mu: \mathbb{D} \rightarrow \mathbb{H}$.
$$\mu(z)=i\frac{1-z}{1+z}$$ 
Then up to a real multiplicative constant,   

$$\mu^*(R'(z)dz)=i^{n+1}\frac{\prod_{j=1}^{m}(1+\xi_j)^2}{\prod_{i=1}^n(1+z_i)} (z+1)^{2m-n-2}\frac{\prod_{k=1}^{n}(z-z_k)}{\prod_{j=1}^{m}(z-\xi_j)^2}dz $$

\begin{thm}
Given $R(z) \in \mathcal{CR}_{m,n}(\boldsymbol{z})$ associate to it a vector field $v_R$ on $\widehat{\mathbb{C}}$ defined by
\begin{equation}
 v_R(z)=\frac{1}{R'(z)} 
\end{equation}    
 The flow lines of $\dot{z} =  v_R(z)$ starting at the critical points are the real locus of $R(z)$.
\end{thm}

\begin{remark}
    This lemma provides an elementary way to plot the real locus of $R(z)$, which are the flow lines of $v_R(z)$ starting at critical points of $R(z)$.
\end{remark}

The critical points are red, the poles are yellow and the marked point $u=-1$ is green.

 \begin{figure}[h]
\begin{minipage}[t]{0.43\linewidth}
    \centering
    \includegraphics[width=6cm]{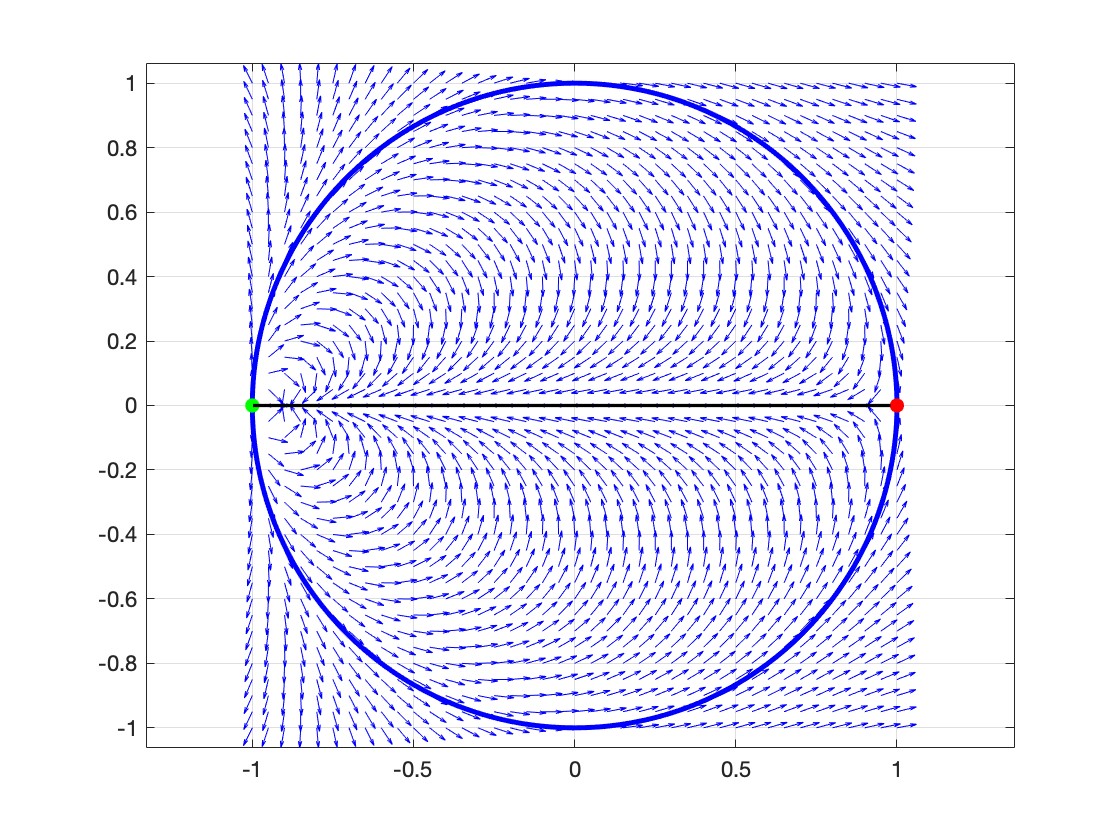}
    \caption{$z_1=1$}
    
\end{minipage}
\begin{minipage}[t]{0.43\linewidth}
    \centering
	\includegraphics[width=6cm]{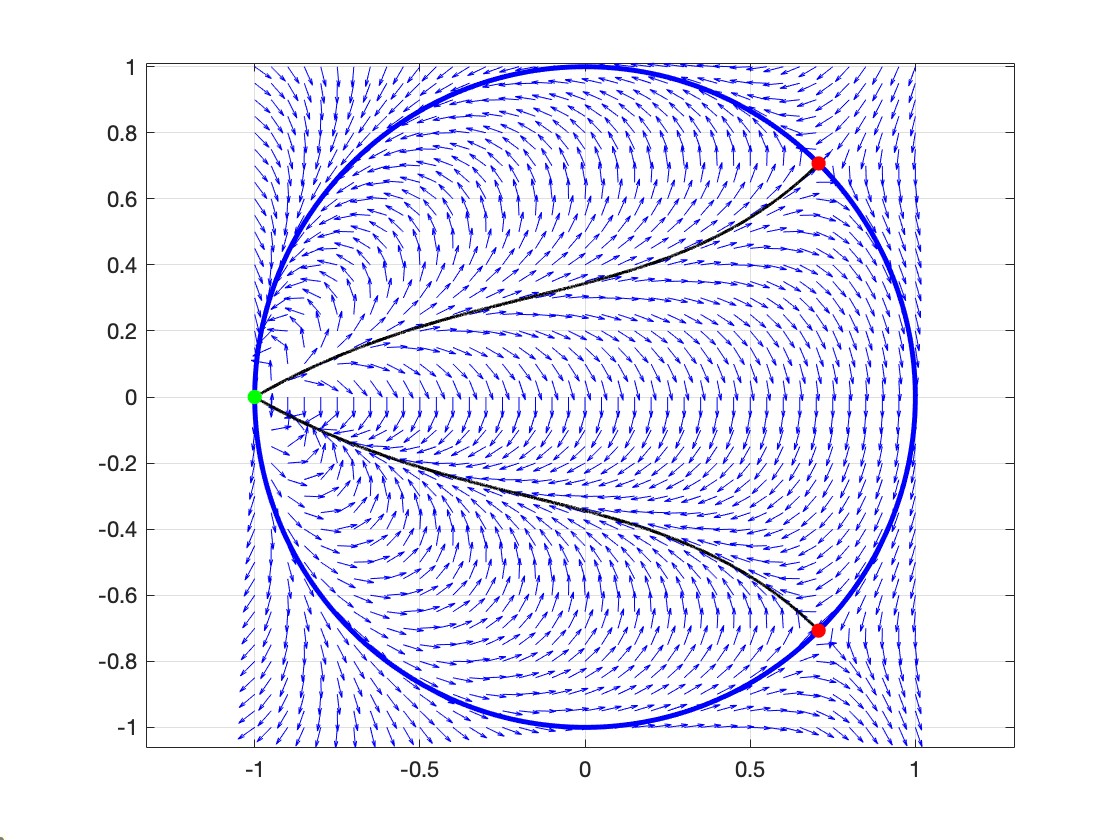}
    \caption{$z_1=e^{i\pi/4},z_2=e^{-i\pi/4},\xi_1=1$}
    
\end{minipage}
\end{figure}
In Figure 4.1, $n=1$,$m=0$,$z_1=1$ the SLE(0) curve connects $z_1$ to $u=-1$.
$$R'(z) = \frac{z-1}{(z+1)^3} $$
thus, as the primitive of $R'(z)$
 $$R(z)=-\frac{z}{(z+1)^2}+c $$
 where $c$ is a real constant.

In Figure 4.2, $n=2$,$m=0$,$z_1=e^{\frac{\pi i}{4}}$,$z_2=e^{\frac{-\pi i}{4}}$, the SLE(0) curves connect $z_1$ to $u=-1$,  $z_2$ to $u=-1$,
$$R'(z) = i\frac{(z-e^{\frac{\pi i}{4}})(z-e^{\frac{-\pi i}{4})})}{(z+1)^4} $$
thus, as the primitive of $R'(z)$
 $$R(z)=\frac{i (-4 + \sqrt{2} + 3 (-2 + \sqrt{2}) z - 6 z^2)}{6 (1 + z)^3}- \frac{i(1-\sqrt{2})}{12}+c $$
 where $c$ is a real constant.
\begin{figure}[h]
\begin{minipage}[t]{0.43\linewidth}
    \centering
    \includegraphics[width=6cm]{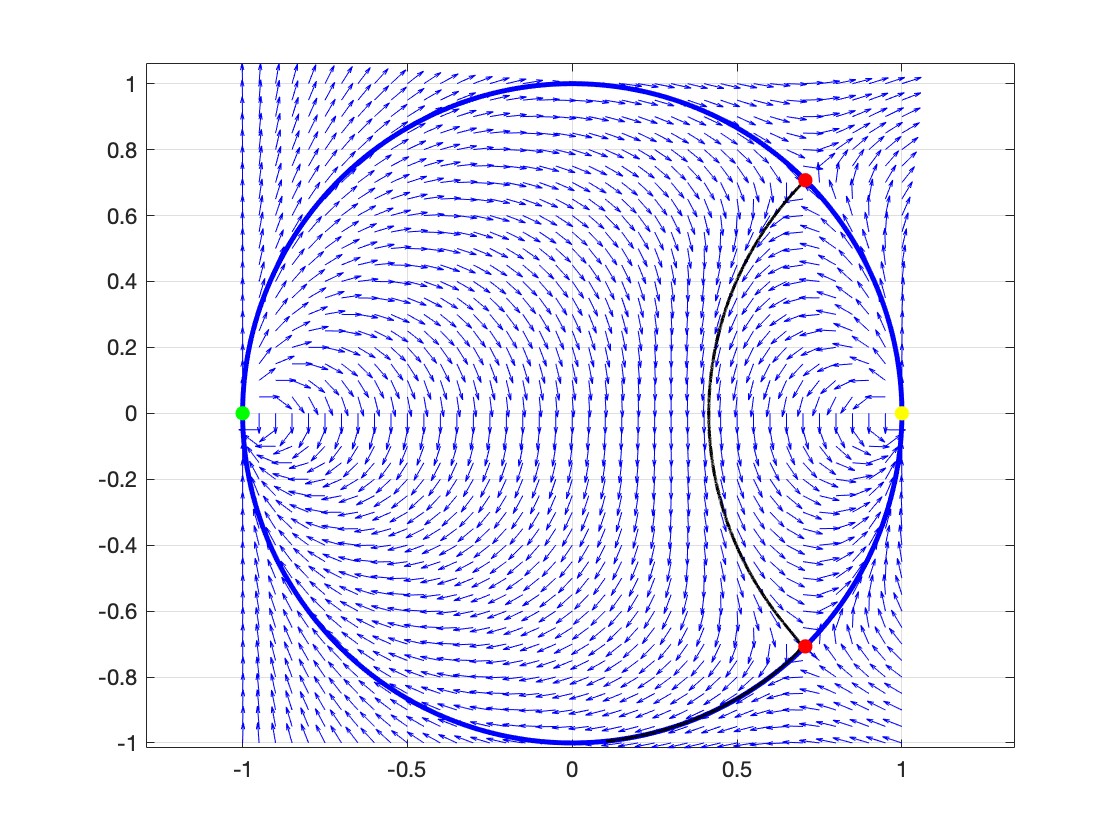}
    \caption{$z_1=e^{i\pi/4},z_2=e^{-i\pi/4},\xi_1=1$}
    
\end{minipage}
\begin{minipage}[t]{0.43\linewidth}
    \centering
	\includegraphics[width=6cm]{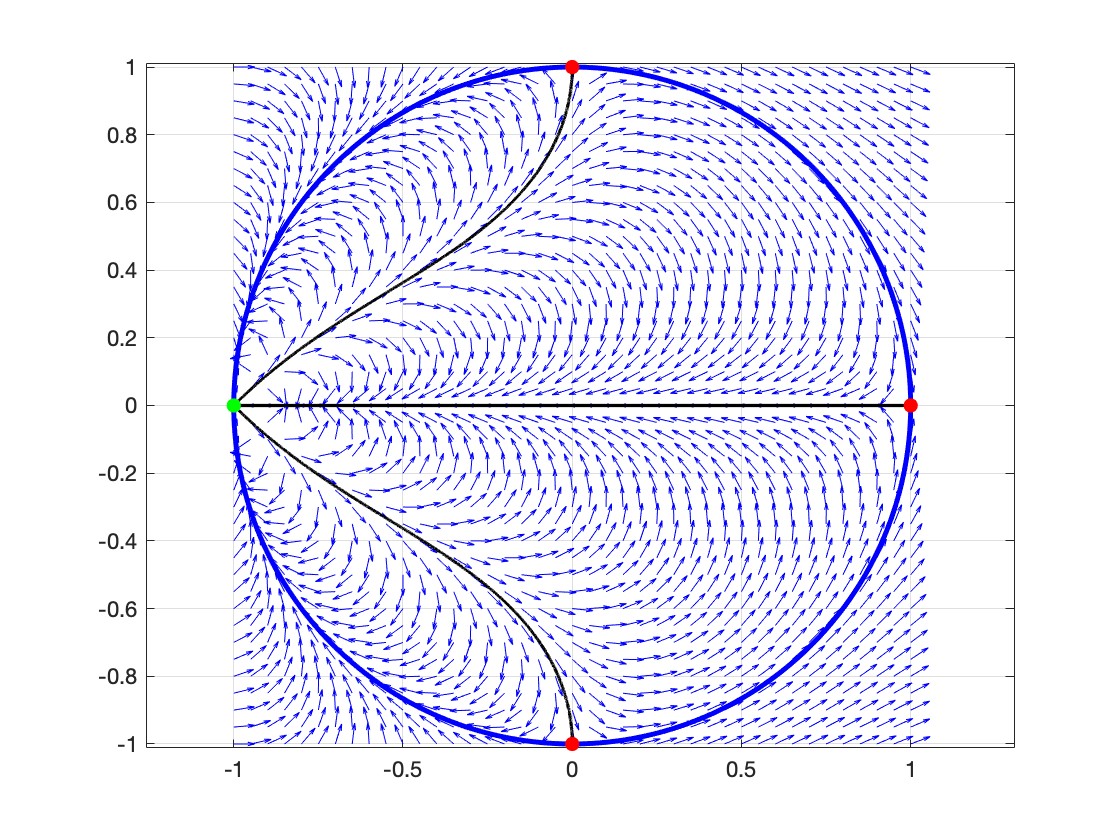}
    \caption{$z_1=i,z_2=1,z_3=-i$}
    
\end{minipage}
\end{figure}

In Figure 4.3, $n=2$,$m=1$,$z_1=e^{\frac{\pi i}{4}}$,$z_2=e^{\frac{-\pi i}{4}}$,$\xi_1=1$ the SLE(0) curve connects $z_1$ to $z_2$.

$$R'(z) = \frac{i (z - e^{\pi i/4}) (z - e^{-\pi i/4})}{(z + 1)^2 (z - 1)^2} $$
thus, as the primitive of $R'(z)$
 $$R(z)= \frac{i(\sqrt{2}-2z)}{2(z^{2}-1)}+ \frac{2+\sqrt{2}i}{4}+c $$
 where $c$ is a real constant.

In Figure 4.4, $n=3$,$m=0$,$z_1=1$,$z_2=i$,$z_3=-i$, the SLE(0) curves connect $z_1$,$z_2$ and $z_3$ to $u=-1$.
$$R'(z) =\frac{(z-i)(z+i)(z-1)}{(z+1)^5} $$
thus, as the primitive of $R'(z)$
 $$R(z)=-\frac{z(1+z+z^2)}{(z+1)^4}+c $$
 where $c$ is a real constant.
\begin{figure}[h]
\begin{minipage}[t]{0.43\linewidth}
    \centering
    \includegraphics[width=6cm]{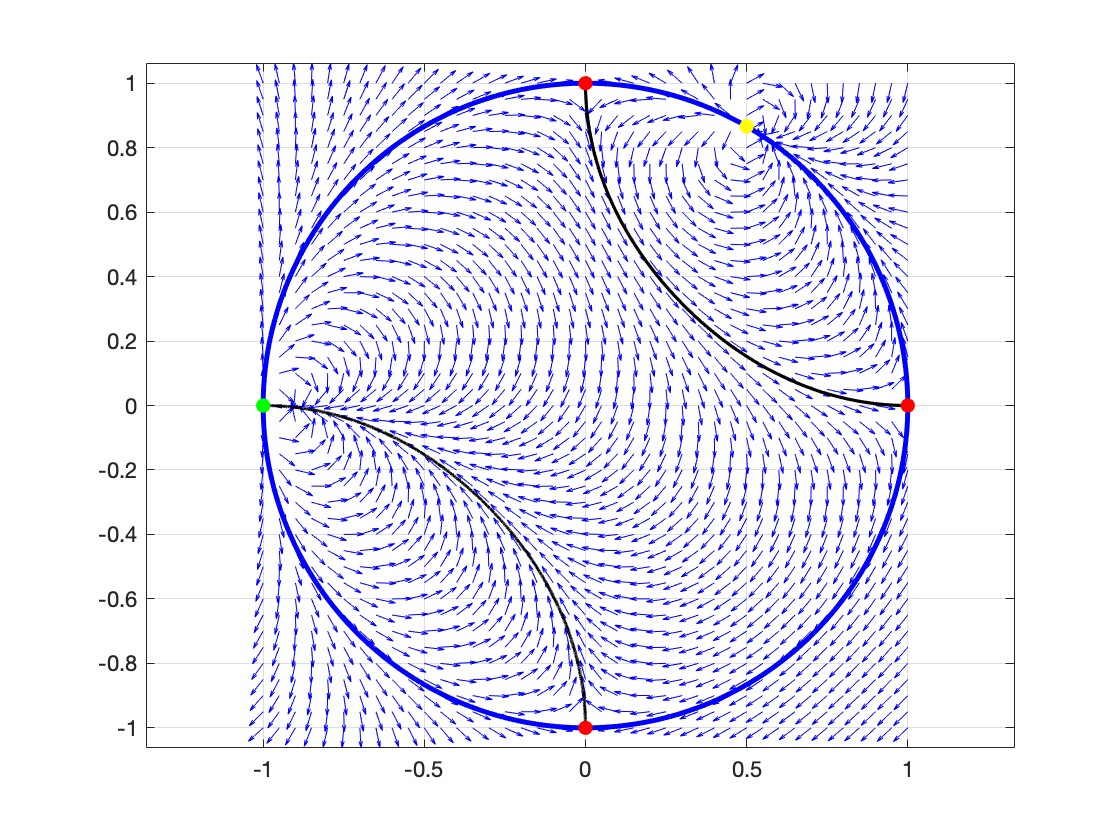}
    \caption{$z_1=i$,$z_2=1$,$z_3=-i$,$\xi_1= e^{\frac{\pi i}{3}}$}
    \label{30}
\end{minipage}
\begin{minipage}[t]{0.43\linewidth}
    \centering
	\includegraphics[width=6cm]{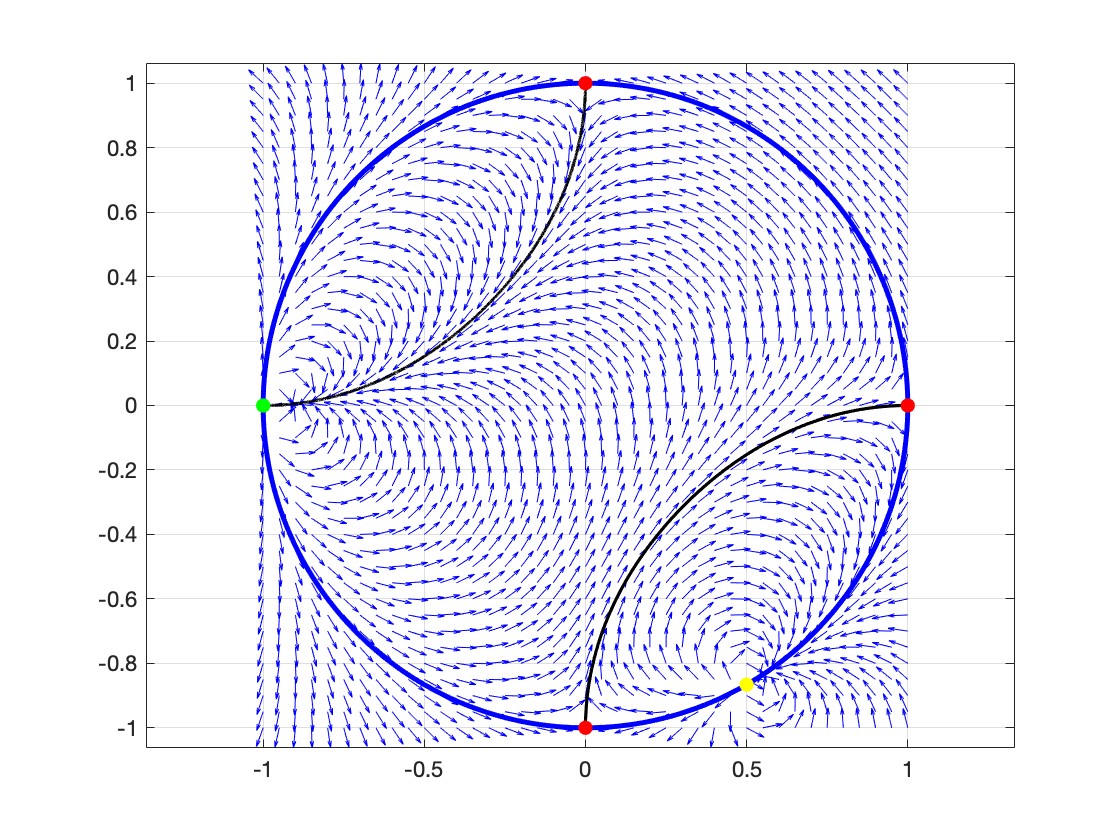}
    \caption{$z_1=i$,$z_2=1$,$z_3=-i$,$\xi_1= e^{-\frac{\pi i}{3}}$}
    \label{31a}
\end{minipage}
\end{figure}

In Figure 4.5, $n=4$,$m=1$,$z_1=i$,$z_2=1$,$z_3=-i$,$\xi_1= e^{\frac{\pi i}{3}}$. The SLE(0) curves connect $z_1$ to $z_2$ and $z_3$ to $u=-1$.

$$R'(z) = \frac{(1+e^{\pi i /3})^2(z-i)(z+i)(z-1)}{(z+1)^3(z-e^{\pi i /3})^2}$$

In Figure 4.6, $n=4$,$m=1$,$z_1=i$,$z_2=1$,$z_3=-i$,$\xi_1= e^{-\frac{\pi i}{3}}$, the SLE(0) curves connect $z_2$ to $z_3$ and connect $z_1$ to $u=-1$.
$$R'(z)=\frac{(1+e^{-\pi i /3})^2(z-i)(z+i)(z-1)}{(z+1)^3(z-e^{-\pi i /3})^2}$$

\begin{figure}[h]
\begin{minipage}[t]{0.43\linewidth}
    \centering
    \includegraphics[width=6cm]{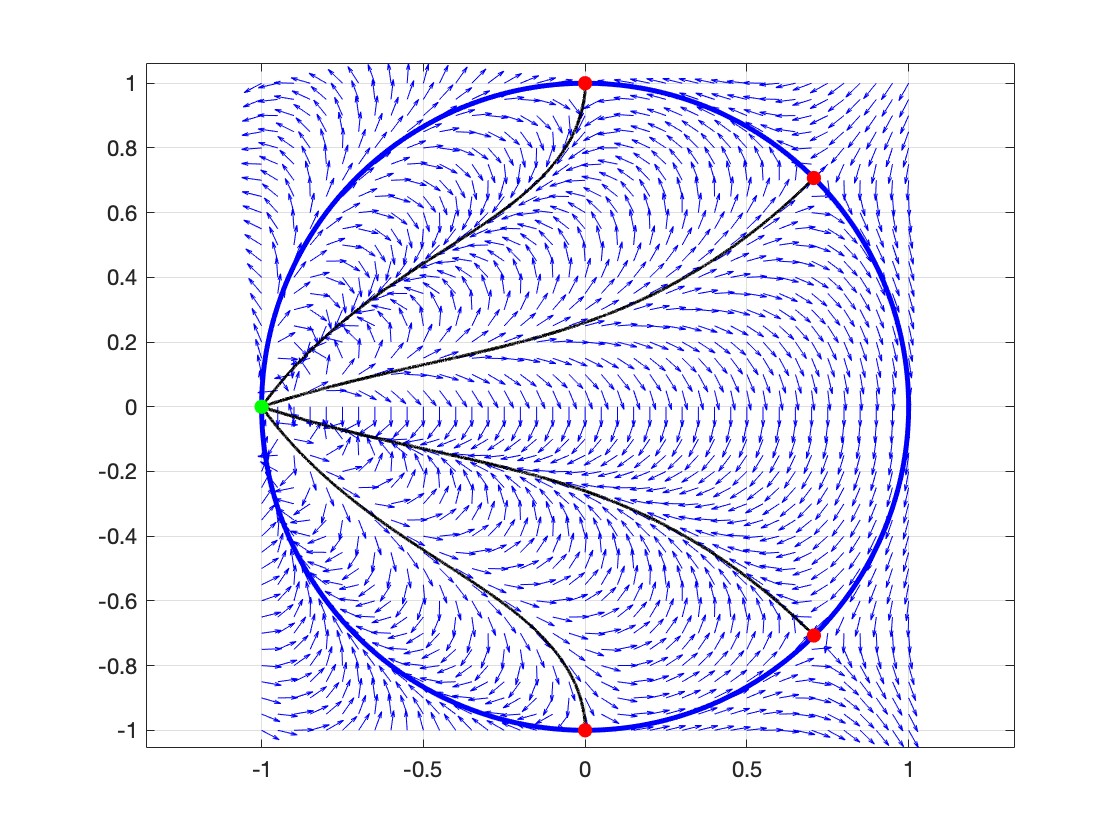}
    \caption{$z_1=i$,$z_2=e^{\frac{\pi i}{4}}$,$z_3=e^{\frac{-\pi i}{4}}$,$z_4=-i$}
    
\end{minipage}
\begin{minipage}[t]{0.43\linewidth}
    \centering
	\includegraphics[width=6cm]{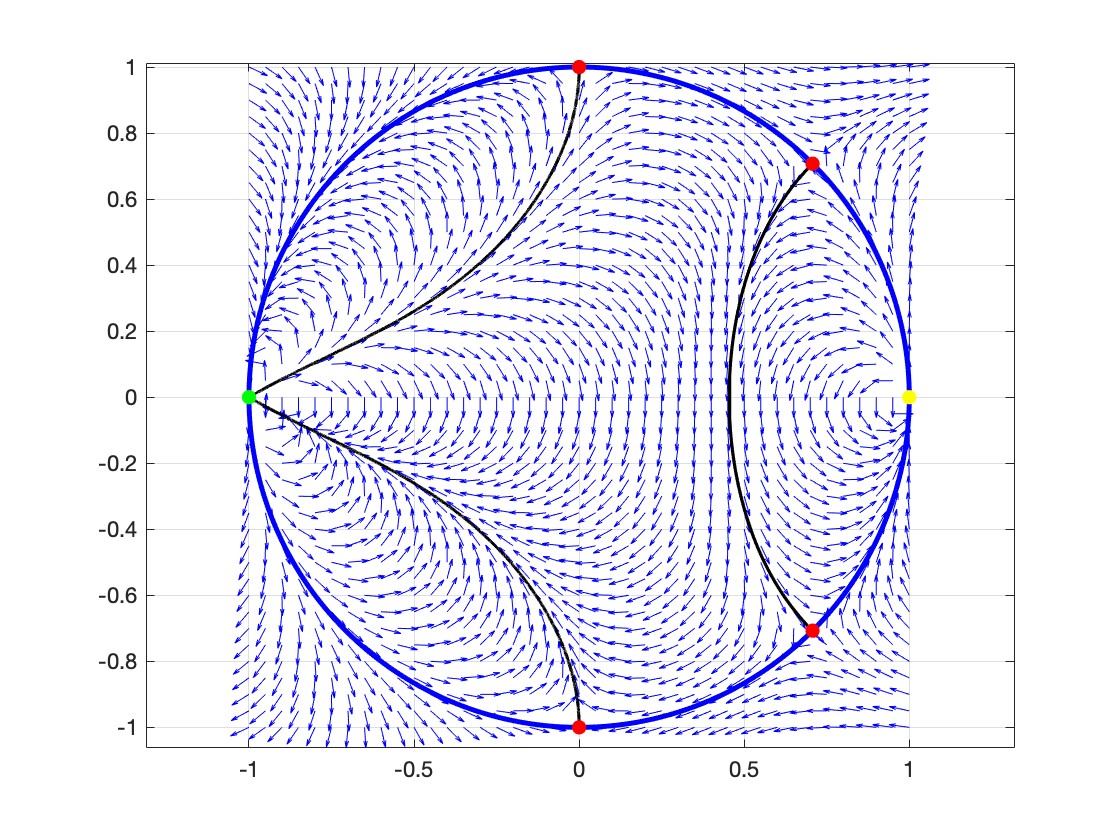}
    \caption{$x_k=e^{\frac{(2k+1)\pi i}{4}}, k=1,2,3,4, \xi_1=1$}
    
\end{minipage}

\end{figure}
\newpage 
In Figure 4.7, $n=4$,$m=0$,$z_1=i$,$z_2=e^{\frac{\pi i}{4}}$,$z_3=e^{\frac{-\pi i}{4}}$,$z_4=-i$. The SLE(0) curves connect $z_1$, $z_2$, $z_3$ and $z_4$ to $u=-1$.

$$R'(z) = \frac{i(z-e^{\pi i/4})(z-e^{-\pi i/4})(z-i)(z+i)}{(z+1)^6} $$

In Figure 4.8, $n=4$,$m=2$,$z_1=i$,$z_2=e^{\frac{\pi i}{4}}$,$z_3=e^{\frac{-\pi i}{4}}$,$z_4=-i$, $\xi_1=1$, the SLE(0) curves connects $z_2$ to $z_3$, connects $z_1$ and $z_4$ to $u=-1$.
$$R'(z) =\frac{i(z-e^{\pi i/4})(z-e^{-\pi i/4})(z-i)(z+i)}{(z+1)^4(z-1)^2} $$

\begin{figure}
   \begin{minipage}[t]{0.43\linewidth}
    \centering
	\includegraphics[width=6cm]{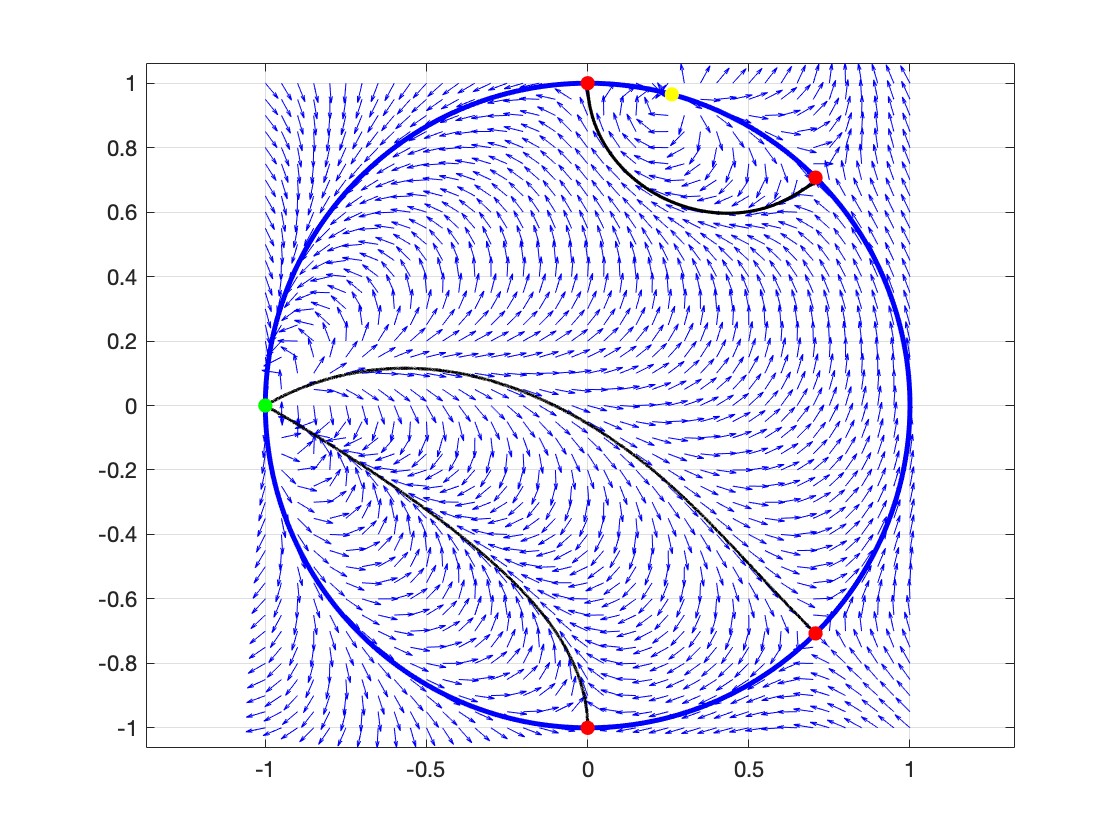}
    \caption{$z_1=i$,$z_2=e^{\frac{\pi i}{4}}$,$z_3=e^{\frac{-\pi i}{4}}$,$z_4=-i$,$\xi_1=\frac{1+2\sqrt{\sqrt{2}+2} i}{2\sqrt{2}+1}$}
    
\end{minipage}
   \begin{minipage}[t]{0.43\linewidth}
    \centering
	\includegraphics[width=6cm]{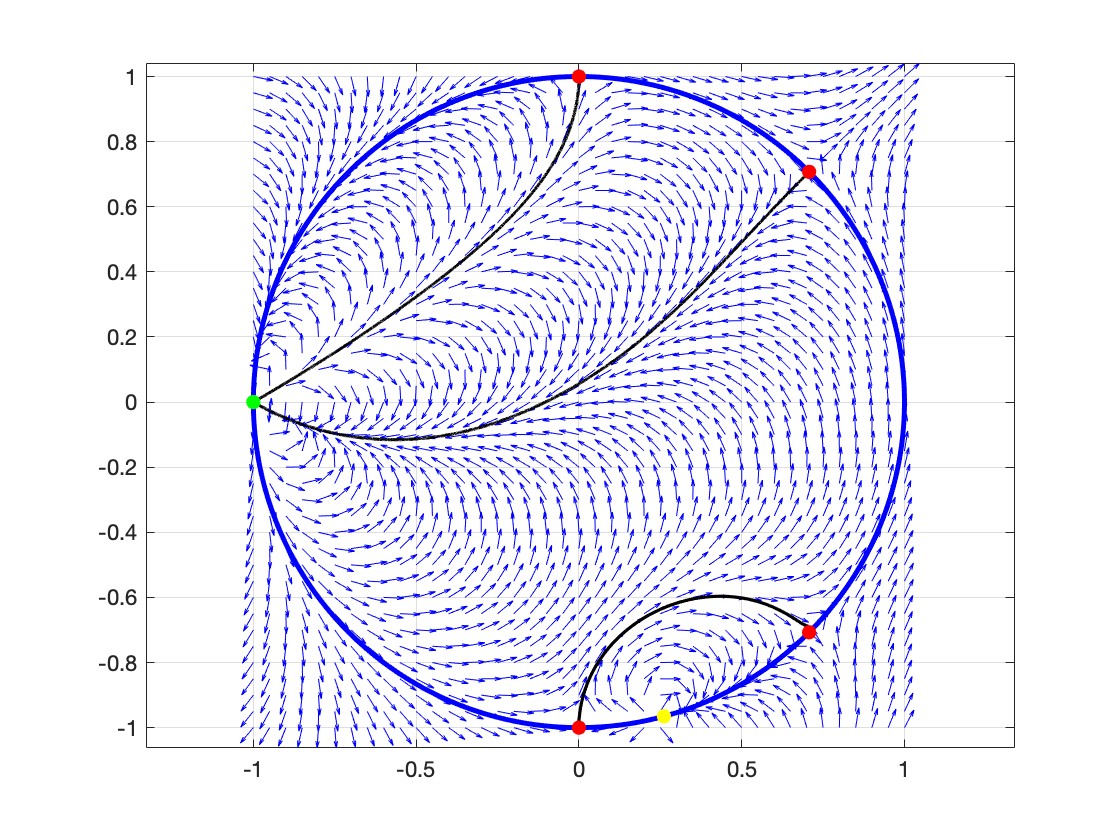}
    \caption{$z_1=i$,$z_2=e^{\frac{\pi i}{4}}$,$z_3=e^{\frac{-\pi i}{4}}$,$z_4=-i$, $\xi_1=\frac{1-2\sqrt{\sqrt{2}+2} i}{2\sqrt{2}+1}$}
    
\end{minipage}
\end{figure}

In Figure 4.9, $n=4$,$m=2$,$z_1=i$,$z_2=e^{\frac{\pi i}{4}}$,$z_3=e^{\frac{-\pi i}{4}}$,$z_4=-i$,$\xi_1=\frac{1+2\sqrt{\sqrt{2}+2} i}{2\sqrt{2}+1}$. The SLE(0) curves connect $z_1$ to $z_2$, connect $z_3$ and $z_4$ to $u=-1$.

$$R'(z) = (\frac{(-2\sqrt{\sqrt{2}+2}+i )}{2\sqrt{2}+1}+1)^2 \frac{(z-e^{\pi i/4})(z-e^{-\pi i/4})(z-i)(z+i)}{(z+1)^4(z-\frac{1+2\sqrt{\sqrt{2}+2} i}{2\sqrt{2}+1})^2} $$

In Figure 4.10, $n=4$,$m=2$,$z_1=i$,$z_2=e^{\frac{\pi i}{4}}$,$z_3=e^{\frac{-\pi i}{4}}$,$z_4=-i$, $\xi_1=\frac{1-2\sqrt{\sqrt{2}+2} i}{2\sqrt{2}+1}$. The SLE(0) curves connect $z_3$ to $z_4$, connect $z_1$ and $z_2$ to $u=-1$.
$$R'(z) =  (\frac{(2\sqrt{\sqrt{2}+2}+i )}{2\sqrt{2}+1}+1)^2 \frac{(z-e^{\pi i/4})(z-e^{-\pi i/4})(z-i)(z+i)}{(z+1)^4(z-\frac{1-2\sqrt{\sqrt{2}+2} i}{2\sqrt{2}+1})^2}$$

\begin{figure}[h]
\begin{minipage}[h]{0.43\linewidth}
    \centering
    \includegraphics[width=6cm]{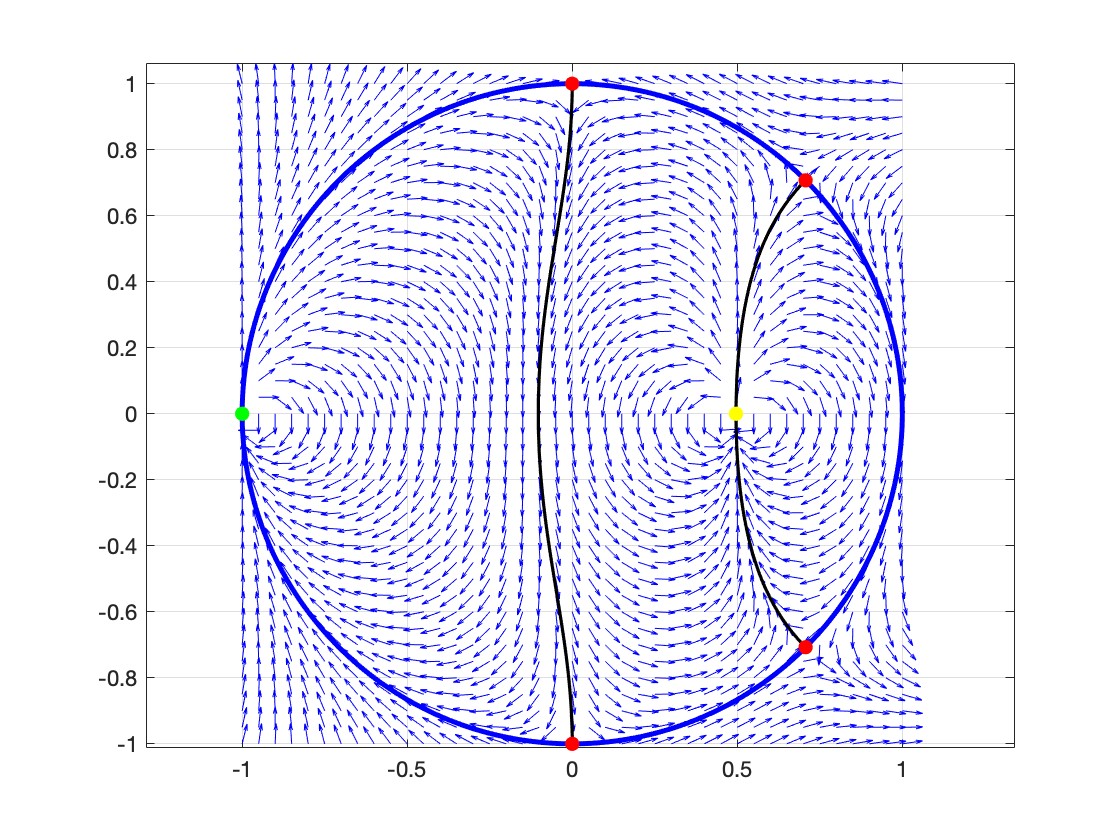}
    \caption{$z_1=i$,$z_2=e^{\frac{\pi i}{4}}$,$z_3=e^{\frac{-\pi i}{4}}$,$z_4=-i$, $\xi_1=0.49604$, $\xi_2=2.0160$}
    
\end{minipage}
\begin{minipage}[h]{0.43\linewidth}
    \centering
	\includegraphics[width=6cm]{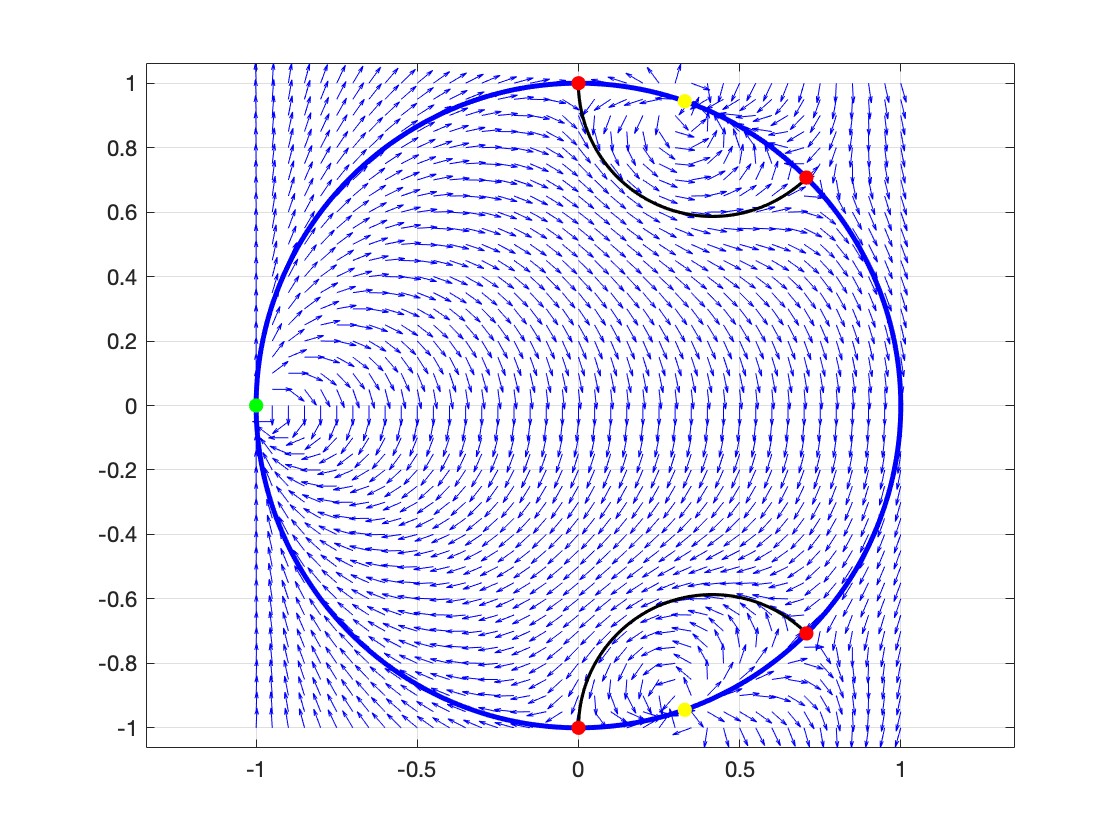}
    \caption{$z_1=i$,$z_2=e^{\frac{\pi i}{4}}$,$z_3=e^{\frac{-\pi i}{4}}$,$z_4=-i$, $\xi_1=0.32979+0.94405i$,$\xi_2=0.32979-0.94405i$}
    
\end{minipage}
\end{figure}

In Figure 4.11, $n=4$,$m=2$,$z_1=i$,$z_2=e^{\frac{\pi i}{4}}$,$z_3=e^{\frac{-\pi i}{4}}$,$z_4=-i$, $\xi_1=0.49604$, $\xi_2=2.0160$. The SLE(0) curve connects $z_1$ and $z_4$, $z_2$ and $z_3$.

$$R'(z) = 
 \frac{i(z-e^{\pi i/4})(z-e^{-\pi i/4})(z-i)(z+i)}{(z+1)^2(z-0.49604)^2(z-2.0160)^2} $$

In Figure 4.12, $n=4$,$m=2$,$z_1=i$,$z_2=e^{\frac{\pi i}{4}}$,$z_3=e^{\frac{-\pi i}{4}}$,$z_4=-i$, $\xi_1=0.32979+0.94405i$, $\xi_2=0.32979-0.94405i$ the SLE(0) curve connects $z_1$ and $z_2$ and connects $z_3$ to $z_4$.
$$R'(z) =\frac{i(z-e^{\pi i/4})(z-e^{-\pi i/4})(z-i)(z+i)}{(z+1)^2 (z-0.32979+0.94405i)^2(z- 0.32979 - 0.94405i)^2)}
$$
\newpage

\subsection{Examples: overscreening}
\label{overscreening}
\indent

\begin{figure}[h]
\begin{minipage}[h]{0.43\linewidth}
    \centering
    \includegraphics[width=6cm]{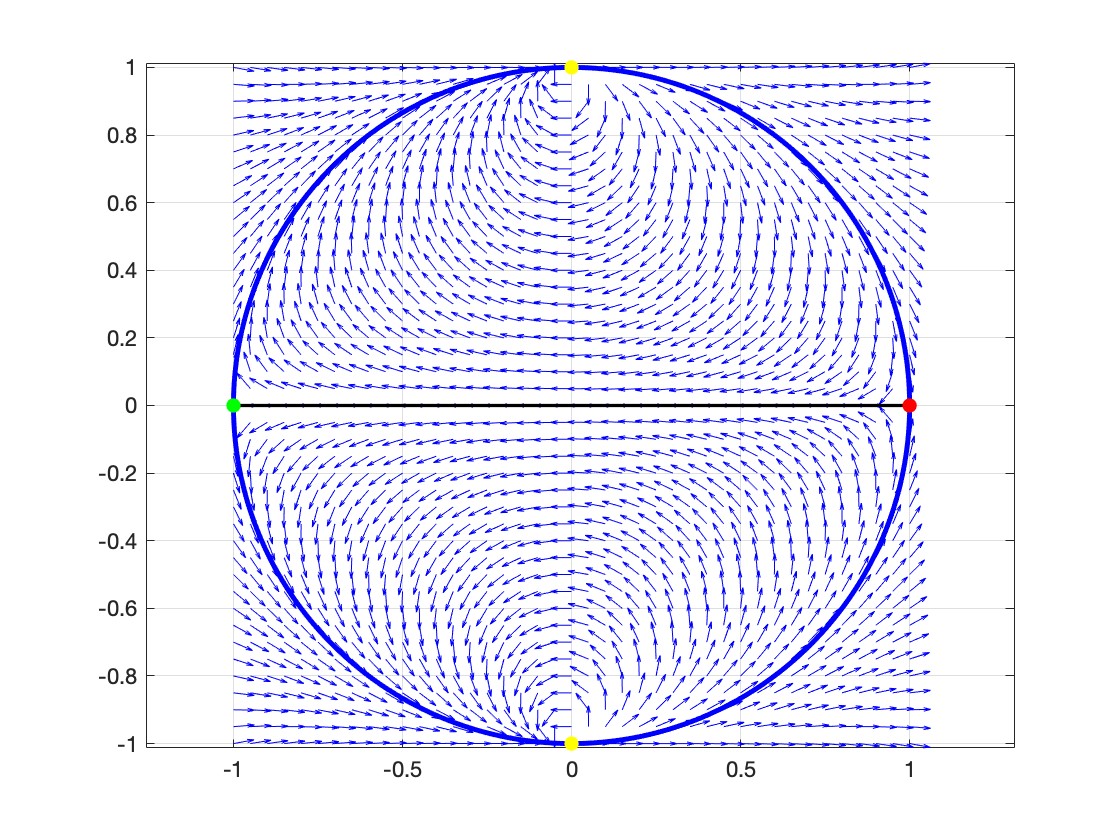}
    \caption{$z_1=1$, $\xi_1=i$, $\xi_2=-i$}
   
\end{minipage}
\begin{minipage}[h]{0.43\linewidth}
    \centering
	\includegraphics[width=6cm]{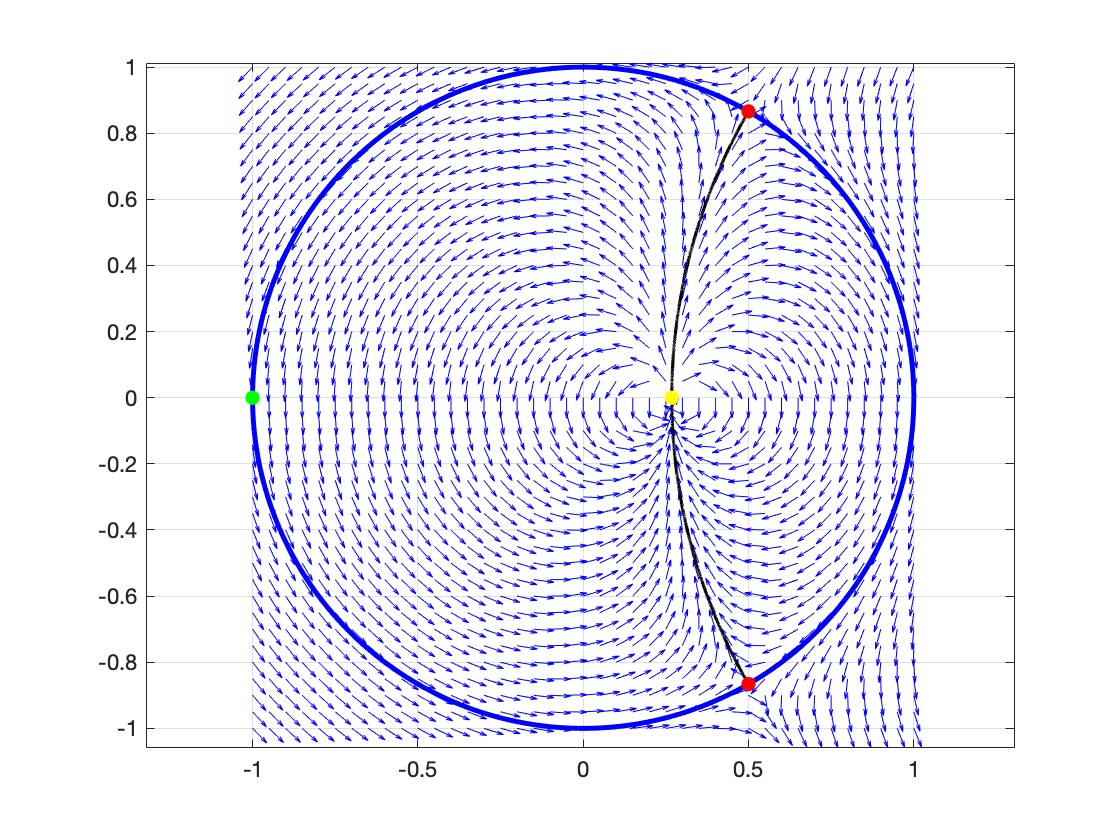}
    \caption{$z_1=e^{\frac{\pi i}{3}}$,$z_2=e^{\frac{-\pi i}{3}}$,$\xi_1=2-\sqrt{3}$,$\xi_2=2+\sqrt{3}$}
   
\end{minipage}

\begin{minipage}[h]{0.43\linewidth}
    \centering
	\includegraphics[width=6cm]{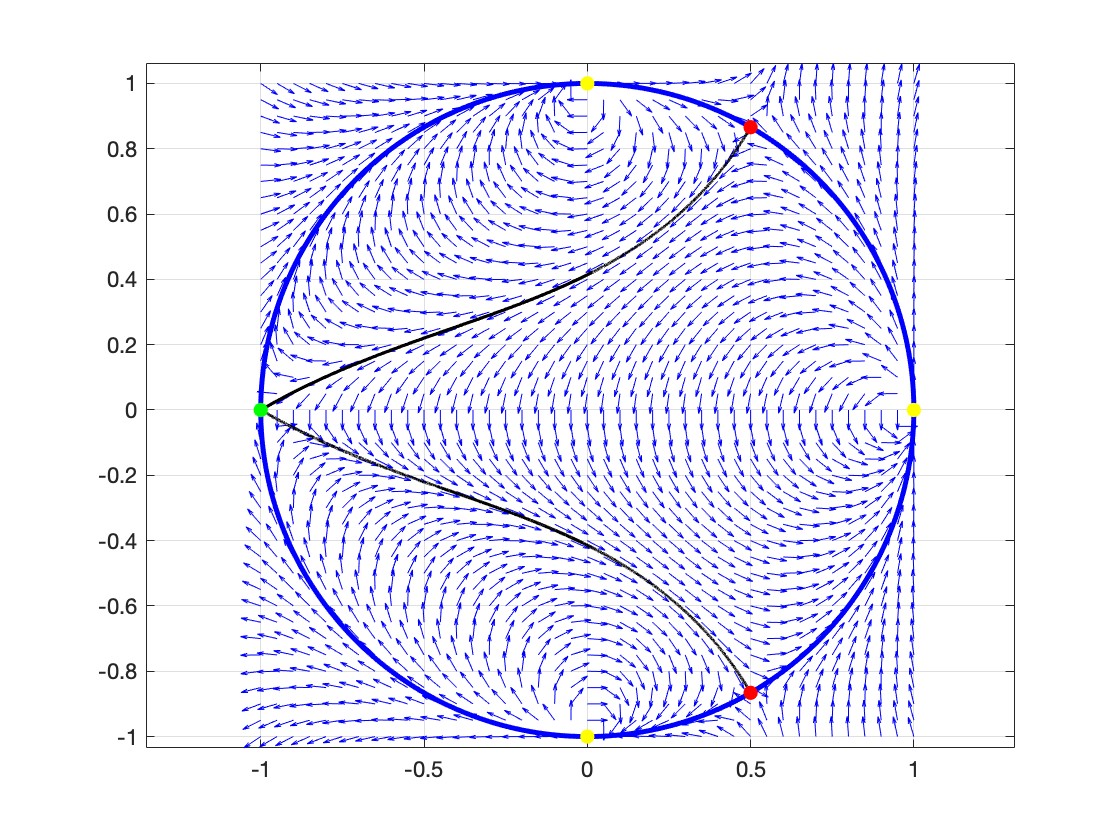}
    \caption{$z_1=e^{\frac{\pi i}{3}}$,$z_2=e^{\frac{-\pi i}{3}}$,$\xi_1=1$,$\xi_2=i$,$\xi_2=-i$}
   
\end{minipage}
\end{figure}

In Figure 4.13, $n=1$,$m=2$,$z_1=1$, $\xi_1=i$, $\xi_2=-i$. The SLE(0) curve connects $z_1$ and $u=-1$.

$$R'(z) = \frac{(z-1)(z+1)}{(z-i)^2(z+i)^2} $$

In Figure 4.14, $n=2$,$m=2$,$z_1=e^{\frac{\pi i}{3}}$,$z_2=e^{\frac{-\pi i}{3}}$,$\xi_1=2-\sqrt{3}$,$\xi_2=2+\sqrt{3}$. The SLE(0) curve connects $z_1$ and $z_2$.
$$R'(z) =\frac{i(z-e^{\pi i/3})(z-e^{-\pi i/3})}{(z+1)^2 (z-2+\sqrt{3})^2(z- 2 - \sqrt{3})^2)}
$$

In Figure 4.15, $n=2$,$m=3$,$z_1=e^{\frac{\pi i}{3}}$,$z_2=e^{\frac{-\pi i}{3}}$,$\xi_1=1$,$\xi_2=i$,$\xi_2=-i$. The SLE(0) curve connects $z_1$ and $z_4$, $z_2$ and $z_3$.

$$R'(z) = 
 \frac{i(z-e^{\pi i/4})(z-e^{-\pi i/4})(z+1)^2}{(z-1)^2(z+i)^2(z-i)^2} $$

\section{Relations to Calogero-Moser system}
\subsection{Multiple  SLE(0) and classical Calogero-Moser system}
\indent

In this section, we investigate the relations between the multiple chordal SLE(0) and classical Calogero-Moser system. 

\begin{defn}
The time evolution of a $n$-particle system on the circle is given by the Hamilton's equations:
$$\dot{x}_j= \frac{\partial{H}}{\partial{p_j}}, \dot{p}_j=-\frac{\partial{H}}{\partial{x_j}}, j=1,\ldots,n, $$
The Hamiltonian is of the form:
$H(\boldsymbol{x},\boldsymbol{p}) = \sum \frac{p_{j}^{2}}{2}+ U(\boldsymbol{x})$ where $U(\boldsymbol{x})$ is a smooth real-valued function on $R^n$.  The initial state of the system is encoded in a position vector $\boldsymbol{x}=(x_1,...,x_n) \in (\mathbb{R}/2\pi \mathbb{Z})^n$ and a momentum vector $p=(p_1,...,p_n)\in \mathbb{R}^n$ where $p_j = \dot{x}_j$, $j=1,2\ldots,n$. \\
We consider the special case where $U(\boldsymbol{x})$ is a sum of pair potentials
$$U(\boldsymbol{x}) = \sum_{j<k} V_{jk}(x_j-x_k).$$
For the Calogero-Moser system, the pair potential is given by $$V_{jk}= -\frac{8}{(x_j-x_k)^2}.$$

\end{defn}

The following theorem describes the evolution of the growth points $\boldsymbol{x}$ and $\boldsymbol{\xi}$ for the multiple chordal SLE(0) systems constructed using stationary relations.
\begin{thm}[See proposition 6.1, Corollary 6.2 in \cite{ABKM20}] \label{Evolution of growth points and screening charges}
 Let $\boldsymbol{x}=\left\{x_1, \ldots, x_{n}\right\}$ be distinct points on the real line and $\boldsymbol{\xi}=\left\{\xi_1, \ldots, \xi_{m}\right\}$ closed under conjugation and solve the stationary relation. Let $\boldsymbol{x}(t)$ and $\boldsymbol{\xi}(t)$ evolve according to multiple chordal SLE(0) system with common parametrization of capacity ($\nu_j(t)=1$). 
 
\begin{itemize}
\item[(i)]  The pair $(\boldsymbol{x}(t), \boldsymbol{\xi}(t))$ forms the closed dynamical system satisfying
\begin{equation}\label{x derivative}
\dot{x}_j=2 \left(\sum_{k \neq j} \frac{2}{x_j-x_k}-\sum_{k=1}^{m} \frac{2}{x_j-\xi_k} \right),
\end{equation}
and

\begin{equation} \label{xi derivative}
\dot{\xi}_k=2\left(-\sum_{l \neq k}\frac{2}{\xi_k-\xi_l}+\sum_{j=1}^{n}\frac{2}{\xi_k-x_j}\right)
\end{equation}

\item[(ii)]$\boldsymbol{x}(t)$ evolve according to the classical Calegero-Moser Hamiltonian, in other words:
$$
\ddot{x}_j=-\sum_{k \neq j} \frac{8}{(x_j-x_k)^3} .
$$
\item[(iii)]$\xi_k$ follows the second-order dynamics.
\begin{equation}
\ddot{\xi}_k=-\sum_{l \neq k}  \frac{8}{(\xi_k-\xi_l)^3} .
\end{equation}
\item[(iv)]
The energy of the system is given by $$\mathcal{H}(\boldsymbol{x},\boldsymbol{p})=0$$

\end{itemize} 
\end{thm}

The multiple chordal SLE(0) systems we constructed using stationary relations are not known to encompass all possible multiple chordal SLE(0) systems.

\begin{proof}[Proof of theorem (\ref{Evolution of growth points and screening charges})]
\

\begin{itemize}
\item[(i)]
 The evolution of $x_j(t)$ is
$$
\begin{aligned}
\dot{x}_j=&\left(\sum_{k \neq j} \frac{2}{x_j-x_k}-2\sum_{k=1}^{m} \frac{2}{x_j-\xi_k}+\sum_{k \neq j} \frac{2}{x_j-x_k}\right) \\
=&2\sum_{k \neq j} \frac{2}{x_j-x_k}-2\sum_{k=1}^{m} \frac{2}{x_j-\xi_k},
\end{aligned}
$$
On the other hand, since the poles follow the Loewner flow we have $\xi_k(t):=$ $g_t\left(\xi_k(0)\right)$, and therefore
$$
\dot{\xi}_k=\dot{g}_t\left(\xi_k(0)\right)= \sum_{j=1}^{n} \frac{2}{g_t\left(\xi_k(0)\right)-x_j}=\sum_{j=1}^{n} \frac{2}{\xi_k-x_j} .
$$
The stationary relation implies that
$$
\dot{\xi}_k=2\sum_{l \neq k} \frac{2}{\xi_k-\xi_l}=-2\sum_{l \neq k}\frac{2}{\xi_k-\xi_l}+ 2\sum_{j=1}^{n}\frac{2}{\xi_k-x_j},
$$
\item[(ii)]
By differentiating, we have
$$
\ddot{x}_j=-\sum_{k \neq j} \frac{2}{\dot{x}_j-\dot{x}_k}\frac{4}{(x_j-x_k)^2}+\sum_l \frac{2}{\dot{x}_j-\dot{\xi}_l}\frac{4}{(x_j-\xi_l)^2} .
$$
Using the formula (\ref{x derivative}) for $\dot{x}_j, \dot{x}_k$ and the equality (\ref{xi derivative})  for $\dot{\xi}_l$ we obtain
$$
\begin{aligned}
\ddot{x}_j= & -\sum_{k \neq j} \frac{1}{(x_j-x_k)^2}  \left( 2\cdot\frac{2}{x_j-x_k}\right)+\\
&-\sum_{k \neq j}\frac{1}{(x_j-x_k)^2} \left(\sum_{l \neq j, k}\left(\frac{2}{x_j-x_l}-\frac{2}{x_k-x_l}\right)  +\sum_l\left(\frac{2}{\xi_l-x_j}-\frac{2}{\xi_l-x_k}\right) \right)\\ & +\sum_l \frac{2}{(x_j-\xi_l)^2}\left(\sum_{k \neq j} \frac{2}{x_j-x_k}+\sum_{s=1}^m \frac{2}{\xi_s-x_j}+\sum_{s \neq l} \frac{2}{\xi_l-\xi_s}-\sum_{k=1}^{n} \frac{2}{\xi_l-x_k}\right) .
\end{aligned}
$$
Rearranging terms gives
$$
\begin{aligned}
\ddot{x}_j & +\sum_{k \neq j} \frac{8}{(x_j-x_k)^3}- \frac{1}{2}\sum_{k \neq j} \sum_{l \neq j, k} \frac{8}{(x_j-x_k)(x_j-x_l)(x_k-x_l)} \\
& =-\frac{1}{2}\sum_{k \neq j} \sum_l \frac{8}{\left(x_j-x_k\right)\left(x_j-\xi_l\right)\left(x_k-\xi_l\right)}+\\
& +\frac{1}{2}\sum_l \frac{4}{\left(x_j-\xi_l\right)^2}\left(\sum_{k \neq j} \frac{2}{x_j-x_k}+\sum_m 
\frac{2}{\xi_m-x_j}-\sum_{m \neq l} \frac{2}{\xi_l-\xi_m}\right) .
\end{aligned}
$$
The last term on the right hand side used the stationary relation and then use the stationary relation again to obtain
$$
\begin{aligned}
&\frac{1}{2}\sum_l \frac{4}{\left(x_j-\xi_l\right)^2}\left(\sum_{k \neq j} \frac{2}{x_j-x_k}+\sum_m 
\frac{2}{\xi_m-x_j}-\sum_{m \neq l} \frac{2}{\xi_l-\xi_m}\right)  \\
& =\frac{1}{2}\sum_l \frac{4}{\left(x_j-\xi_l\right)^2} \sum_{m \neq l}\left(\frac{2}{\xi_l-\xi_m}+\frac{2}{\xi_m-x_j}\right)\\
&=\frac{1}{2}\sum_l \sum_{m \neq l} \frac{8}{\left(x_j-\xi_l\right)\left(x_j-\xi_m\right)\left(\xi_l-\xi_m\right)} .
\end{aligned}
$$
Combining all of the above, we obtain
$$
\begin{aligned}
\ddot{x}_j+\sum_{k \neq j} \frac{8}{(x_j-x_k)^3} & =\frac{1}{2}\sum_{k \neq j} \sum_{l \neq j, k} \frac{8}{\left(x_j-x_k\right)\left(x_j-x_l\right)\left(x_k-x_l\right)} \\
& +\frac{1}{2}\sum_l \sum_{m \neq l} \frac{8}{\left(x_j-\xi_l\right)\left(x_j-\xi_m\right)\left(\xi_l-\xi_m\right)}
\end{aligned}
$$
The right-hand side is canceled by symmetry.

\item[(iii)]

Differentiating the equality (\ref{xi derivative}), we have
$$
\ddot{\xi}_k=-\sum_{l \neq k} \frac{4(\dot{\xi}_k-\dot{\xi}_l)}{ (\xi_k-\xi_l)^2} .
$$
Now by using the first equality of (6.4) again for $\dot{\xi}_k, \dot{\xi}_l$ we obtain
$$
\ddot{\xi}_k=-\frac{1}{2}\sum_{l \neq k} \frac{4}{(\xi_k-\xi_l)^2}\left(\frac{2}{\xi_k-\xi_l}+\sum_{m \neq k, l} \frac{2}{\xi_k-\xi_m}-\frac{2}{\xi_l-\xi_k}-\sum_{m \neq k, l} \frac{2}{\xi_l-\xi_m}\right)
$$
Rearranging terms gives
$$
\ddot{\xi}_k=-\sum_{l \neq k}  \frac{8}{(\xi_k-\xi_l)^3}+ \frac{1}{2}\sum_{l \neq k} \sum_{m \neq k, l} \frac{8}{(\xi_k-\xi_l)(\xi_k-\xi_m)(\xi_l-\xi_m)} .
$$
The last term is canceled by symmetry.
\item[(iv)]
  For a multiple chordal SLE(0) system with $n$ growth points and $m$ screening charges that solve the stationary relations, by theorem (\ref{thm stationary relations imply commutation}), 
  $$U_j=\sum_{k \neq j} \frac{2}{x_j-x_k}-2\sum_{k=1}^{m} \frac{2}{x_j-\xi_k}$$
  satisfies the null vector equation (\ref{null vector equation for kappa 0}).
   Plugging into equation (\ref{null Hamiltonian}) and equation (\ref{CS and Null}), we obtain the desired result.
\end{itemize}
\end{proof}

\begin{proof}[Proof of theorem (\ref{CM results kappa=0})]
For multiple chordal SLE(0) system with common parametrization of capacity (i.e. $\nu_j(t)=1$ for $j=1,2,\ldots,n$), let $\left\{\left(x_j, U_j\right), j=1, \ldots,  n\right\}$ are related to $\left\{\left(x_j, p_j\right), j=1, \ldots,  n\right\}$ via
$$
p_j=\left(U_j+\sum_{k \neq j}\frac{2}{x_j-x_k} \right)
$$
where $U_j$ solves the null vector equations (\ref{null vector equation for kappa 0}).
\begin{itemize}
 
\item[(i)]

Solving for $U_j$ and inserting the result into the left-hand side of the null vector equation leads to the identity.

\begin{equation}\label{null Hamiltonian}
\begin{aligned}
h=&\frac{1}{2} U_j^2+\sum_k f_{k j} U_k-\sum_k \frac{3}{2}f_{ jk}^2 \\
&=\frac{1}{2} p_j^2- \sum_k\left(p_j+p_k\right) f_{j k}+\sum_k \sum_{l \neq k} f_{j k} f_{j l}-2\sum_k f_{j k}^2 \\
&=\mathcal{H}_j(\boldsymbol{x}, \boldsymbol{p})
\end{aligned} 
\end{equation}

where 
$$
f_{j k}=f_{j k}(\boldsymbol{x})= \begin{cases}0, & j=k \\ \frac{2}{x_j-x_k}, & j \neq k\end{cases}
$$

Therefore, $\mathcal{H}_j$ is preserved under the Loewner flow.

Futheremore, for each $c \in \mathbb{R}$, the submanifolds defined by the null vector Hamiltonian
\begin{equation}
N_c=\left\{(\boldsymbol{x}, \boldsymbol{p}): \mathcal{H}_j(\boldsymbol{x}, \boldsymbol{p})=c \text { for all } j\right\}
\end{equation}
are invariant under the Loewner flow.

By direct computation, $\mathcal{H}_j$ is related to the Calogero-Moser Hamiltonian $\mathcal{H}$ by:
\begin{equation} \label{CS and Null}
 \sum_j \mathcal{H}_j=\mathcal{H}   
\end{equation}

Our next result shows that null vector Hamiltonian $\mathcal{H}_j$ has a nice interpretation in terms of the Lax pair for the Calogero-Moser system.

\begin{thm}
The Lax pair is two square matrices $L=L(\boldsymbol{x}, \boldsymbol{p})$ and $M=M(\boldsymbol{x}, \boldsymbol{p})$ each of size $ n \times  n$, and by \cite{Mos75} the entries are given by
$$
L_{j k}=\left\{\begin{array}{ll}
p_j, & j=k, \\
2f_{j k}, & j \neq k,
\end{array} \quad \text { and } \quad M_{j k}= \begin{cases}-\sum_l f_{j l}^2, & j=k \\
f_{j k}^2, & j \neq k\end{cases}\right.
$$
This leads to the following representation of $\mathcal{H}_j$ in terms of $L^2$.
\begin{equation}
\mathcal{H}_j=\frac{1}{2} \mathrm{e}_j^{\prime} L^2 \mathbf{1}    
\end{equation}

where $\mathrm{e}_j^{\prime}$ is the transpose of the $j$ th standard basis vector and 1 is the vector of all ones. 

Consequently, the $U_j, j=1, \ldots, n$, defined by solving the null vector equations for a given $\boldsymbol{x}$ iff the $\boldsymbol{p}$ variables satisfy $L^2(\boldsymbol{x}, \boldsymbol{p}) \mathbf{1}=\mathbf{0}$.
\end{thm}
\begin{proof}
Write $L=P-X_1$, where $P=P(\boldsymbol{p})=\operatorname{diag}(\boldsymbol{p})$ is the square matrix with entries of $\boldsymbol{p}$ along its diagonal, and $X_1=X_1(\boldsymbol{x})$ is the square matrix with entries $\left(X_1\right)_{j k}=f_{j k}$. Note that $P$ is symmetric and $X_1$ is anti-symmetric. Then
$$
L^2=P^2-P X_1-X_1 P+X_1^2
$$
It is straightforward to compute the entries of $P^2-P X_1-X_1 P$ and see that they give the first two terms on the right-hand side of the Hamiltonian. For $X_1^2$ we have
$$
\begin{aligned}
\mathrm{e}_j^{\prime} X_1^2 1=\sum_k\left(X_1^2\right)_{j k}&=-4(\sum_k \sum_i f_{j l} f_{k l})\\
& =-4\left(\sum_l f_{j l}^2+\sum_{k\neq j} \sum_{l \neq j} f_{j l} f_{k l}\right) \\
& =-4\left(\sum_l f_{j 1}^2+\frac{1}{2} \sum_{k \neq j } \sum_{l \neq k}\left(f_{j l} f_{k l}+f_{j k} f_{l k}\right)\right) \\
& =-4\left(\sum_l f_{j 1}^2-\frac{1}{2} \sum_{k \neq j } \sum_{l \neq k} f_{j k} f_{j l}\right).
\end{aligned}
$$
\end{proof}
\item[(ii)]\begin{defn}[Poisson Bracket]
For any smooth function \( F = F(\boldsymbol{x}, \boldsymbol{p}) \) defined on phase space, the associated vector field is given by  
\[
X_F = \sum_{j=1}^{n} \frac{\partial F}{\partial p_j} \partial_{x_j} - \sum_{j=1}^{n} \frac{\partial F}{\partial x_j} \partial_{p_j}.
\]  
Given two smooth functions \( F = F(\boldsymbol{x}, \boldsymbol{p}) \) and \( G = G(\boldsymbol{x}, \boldsymbol{p}) \), the commutator of their associated vector fields satisfies  
\[
[X_F, X_G] = X_{\{F, G\}},
\]  
where \( \{F, G\} \), the Poisson bracket of \( F \) and \( G \), is defined by  
\[
\{F, G\} = \sum_{j=1}^{n} \left( \frac{\partial F}{\partial p_j} \frac{\partial G}{\partial x_j} - \frac{\partial F}{\partial x_j} \frac{\partial G}{\partial p_j} \right).
\]  
\end{defn}

By direct computation, for all \( j, k \), the null vector Hamiltonians \( \mathcal{H}_j \) and \( \mathcal{H}_k \) satisfy the Poisson bracket identity  
\[
\{\mathcal{H}_j, \mathcal{H}_k\} = \frac{1}{f_{jk}^2} \left(\mathcal{H}_k - \mathcal{H}_j\right).
\]  

By the definition of \( N_c \), we have \( \{\mathcal{H}_j, \mathcal{H}_k\} = 0 \) along \( N_c \).  

Thus, the vector fields \( X_{\mathcal{H}_j} \) induced by the Hamiltonians \( \mathcal{H}_j \) commute along the submanifolds \( N_c \).
\end{itemize}
\end{proof}

\subsection{Classical Limit and Quantization}
\indent

The relationship between multiple chordal SLE(0) systems and the classical Calogero--Moser system can be understood as the classical limit of the corresponding quantum system. This correspondence manifests clearly through the asymptotic analysis of partition functions as $\kappa \to 0$.

From the viewpoint of partition functions, we expect that for suitably chosen $\mathcal{Z}^\kappa(\boldsymbol{x})$, the following classical limits exist:
\[
\lim_{\kappa \to 0} \frac{\partial \log \mathcal{Z}^\kappa(\boldsymbol{x})}{\partial x_j} = U_j, \qquad
\lim_{\kappa \to 0} \frac{\partial \log \tilde{\mathcal{Z}}^\kappa(\boldsymbol{x})}{\partial x_j} = p_j,
\]
where the pair $(U_j, p_j)$ satisfies the classical Calogero--Moser equations described in~\eqref{CM results kappa=0}.

This correspondence can be verified directly: as $\kappa \to 0$, the logarithmic derivatives of partition functions converge to the classical variables $U_j$ and $p_j$, thereby encoding the classical dynamics.

From the operator-theoretic perspective, the quantum Hamiltonians arise from their classical counterparts via canonical quantization. In this framework, classical position and momentum variables are promoted to differential operators:
\[
x_j \mapsto \hat{x}_j, \qquad
p_j \mapsto \hat{p}_j = \kappa \frac{\partial}{\partial x_j}.
\]
Accordingly, the classical Poisson brackets are replaced by scaled commutators:
\[
\begin{aligned}
\{p_i, x_j\} &= \delta_{ij} \quad \Rightarrow \quad \frac{1}{\kappa}[\hat{p}_i, \hat{x}_j] = \delta_{ij}, \\
\{x_i, x_j\} &= 0 \quad \Rightarrow \quad \frac{1}{\kappa}[\hat{x}_i, \hat{x}_j] = 0, \\
\{p_i, p_j\} &= 0 \quad \Rightarrow \quad \frac{1}{\kappa}[\hat{p}_i, \hat{p}_j] = 0.
\end{aligned}
\]
Under this quantization procedure, the classical Hamiltonians $\mathcal{H}$ and $\mathcal{H}_j$ are mapped to quantum operators $\mathcal{L}$ and $\mathcal{L}_j$, respectively. For further discussion on the quantization of the Calogero--Moser system, see~\cite{E07}.

\section*{Acknowledgement}
I express my sincere gratitude to Professor Nikolai Makarov for his invaluable guidance.


\begin{thebibliography}{1000}


\bibitem[ABKM20]{ABKM20} T. Alberts, S. S. Byun, N-G. Kang, and N. Makarov. \emph{Pole dynamics and an integral of motion for multiple SLE(0)}. Preprint in arXiv:2011.05714, 2020.


\bibitem[AGK11]{AGK11}
A. Abanov, A. Gromov, and M. Kulkarni, \emph{Soliton solutions
  of a {C}alogero model in a harmonic potential}, Journal of Physics A:
  Mathematical and Theoretical \textbf{44} (2011), no.~29, 295203.


\bibitem[AS60]{AS60}
L. V. Ahlfors, L. Sario,
\emph{Riemann Surfaces}, Princeton Universty Press, 1960 .

\bibitem[AHSY23]{AHSY23}
M. Ang, N. Holden, X. Sun, and P. Yu. \emph{Conformal welding of quantum disks and multiple SLE: the non-simple case}. Preprint in arXiv:2310.20583, 2023.

\bibitem[BB03a]{BB03a} 
M. Bauer and D. Bernard. \emph{Conformal field theories of stochastic Loewner evolutions}, Comm. Math. Phys., 239(3):493-521, 2003.



\bibitem[BBK05]{BBK05}
M. Bauer, D. Bernard, and Kalle Kyt{\"o}l{\"a}, \emph{Multiple
  {S}chramm-{L}oewner evolutions and statistical mechanics martingales}, J.
  Stat. Phys. \textbf{120} (2005), no.~5-6, 1125--1163. \MR{2187598}

\bibitem[BE21]{BE:canonical}
M. Bonk and A. Eremenko, \emph{Canonical embeddings of pairs of
  arcs}, Comput. Methods Funct. Theory \textbf{21} (2021), 825--830.

\bibitem[BPW21]{BPW:uniqueness}
V. Beffara, E. Peltola, and Hao Wu, \emph{On the uniqueness of
  global multiple {SLE}s}, Ann. Probab. \textbf{49} (2021), no.~1, 400--434.

\bibitem[Cal71]{Calogero:quantum_CM_integrable}
F.~Calogero, \emph{Solution of the one-dimensional {$N$}-body problems with
  quadratic and/or inversely quadratic pair potentials}, J. Mathematical Phys.
  \textbf{12} (1971), 419--436. \MR{280103}

\bibitem[Car96]{Car96} J. L. Cardy. \emph{Scaling and renormalization in statistical physics}, vol. 5 of Cambridge Lecture Notes in Physics. Cambridge University Press, Cambridge, 1996.

\bibitem[Car04]{Car04}
J. Cardy, \emph{Calogero-Sutherland model and bulk-boundary correlations in
  conformal field theory}, Phys. Lett. B \textbf{582} (2004), no.~1-2,
  121--126. \MR{2047300}

\bibitem[DC07]{DC07}
B.~Doyon and J.~Cardy, \emph{Calogero-{S}utherland eigenfunctions with mixed
  boundary conditions and conformal field theory correlators}, J. Phys. A
  \textbf{40} (2007), no.~10, 2509--2540. \MR{2305181}

\bibitem[Dub06]{Dub06}
Julien Dub{\'e}dat, \emph{Euler integrals for commuting {SLE}s}, J. Stat. Phys.
  \textbf{123} (2006), no.~6, 1183--1218. \MR{2253875}

\bibitem[Dub07]{Dub07}
\bysame, \emph{Commutation relations for {S}chramm-{L}oewner evolutions}, Comm.
  Pure Appl. Math. \textbf{60} (2007), no.~12, 1792--1847. \MR{2358649}

\bibitem[Dub09]{Dub09}
\bysame, \emph{SLE and the free field: partition functions and couplings}, J.
  Amer. Math. Soc. \textbf{22} (2009), no.~4, 995--1054. \MR{2525778}

\bibitem[Dub15a]{Dub15a} \bysame.\emph{SLE and Virasoro representations: localization}, Comm. Math. Phys., 336(2):695-760, 2015.
\bibitem[Dub15b]{Dub15b} \bysame.\emph{SLE and Virasoro representations: fusion}, Comm. Math. Phys., 336(2):761-809, 2015.

\bibitem[E07]{E07}
P. Etingof \emph{Calogero-Moser Systems and Representation Theory},
Zurich lectures in advanced mathematics, European Mathematical Society, 2007.

\bibitem[EG02]{EG02}
A.~Eremenko and A.~Gabrielov, \emph{Rational functions with real critical
  points and the {B}. and {M}. {S}hapiro conjecture in real enumerative
  geometry}, Ann. of Math. (2) \textbf{155} (2002), no.~1, 105--129.
  \MR{1888795}

\bibitem[EG11]{EG11}
\bysame, \emph{An elementary proof of the {B}.
  and {M}. {S}hapiro conjecture for rational functions}, Notions of positivity
  and the geometry of polynomials, Trends Math., Birkh\"{a}user/Springer Basel
  AG, Basel, 2011, pp.~167--178. \MR{3051166}

\bibitem[FK04]{FK04} R. Friedrich and J. Kalkkinen. \emph{On conformal field theory and stochastic Loewner evolution}, Nucl. Phys. B, 687(3):279-302, 2004.

\bibitem[FK15a]{FK15a}
S. Flores and P. Kleban, \emph{A solution space for a system of
  null-state partial differential equations: {P}art 1}, Comm. Math. Phys.
  \textbf{333} (2015), no.~1, 389--434. \MR{3294954}

\bibitem[FK15b]{FK15b}
\bysame, \emph{A solution space for a system of null-state partial differential
  equations: {P}art 2}, Comm. Math. Phys. \textbf{333} (2015), no.~1, 435--481.
  \MR{3294955}

\bibitem[FK15c]{FK15c}
\bysame, \emph{A solution space for a system of null-state partial differential
  equations: {P}art 3}, Comm. Math. Phys. \textbf{333} (2015), no.~2, 597--667.
  \MR{3296159}

\bibitem[FK15d]{FK15d}
\bysame, \emph{A solution space for a system of null-state partial differential
  equations: {P}art 4}, Comm. Math. Phys. \textbf{333} (2015), no.~2, 669--715.
  \MR{3296160}

 
\bibitem[FLPW24]{FLPW24}
Y. Feng, M. Liu, E. Peltola, H. Wu . \emph{Multiple SLEs for $\kappa\in (0, 8) $: Coulomb gas integrals and pure partition functions}. arXiv preprint arXiv:2406.06522, 2024

\bibitem[FW03]{FW03} R. Friedrich and W. Werner. \emph{Conformal restriction, highest-weight representations and SLE}, Comm. Math. Phys., 243(1):105-122, 2003.


\bibitem[GGJ07]{GGJ:phase_portrait}
A. Garijo, A. Gasull, and X. Jarque, \emph{Local and global
  phase portrait of equation {$\dot z=f(z)$}}, Discrete Contin. Dyn. Syst.
  \textbf{17} (2007), no.~2, 309--329. \MR{2257435}

\bibitem[Gol91]{Goldberg91}
L. Goldberg, \emph{Catalan numbers and branched coverings by the {R}iemann
  sphere}, Adv. Math. \textbf{85} (1991), no.~2, 129--144. \MR{1093002}

\bibitem[Gra07]{Graham:multiple_SLE}
K.~Graham, \emph{On multiple {S}chramm-{L}oewner evolutions}, J. Stat. Mech.
  Theory Exp. (2007), no.~3, P03008, 21. \MR{2318432}
\bibitem[GL98]{GL98} J.J. Graham and G.I. Lehrer, \emph{The representation theory of affine Temperley-Lieb algebras}, Enseign. Math., 44:173-218, 1998.

\bibitem[HL21]{HL21}
V. Healey, and G. Lawler. \emph{"N-sided chordal Schramm–Loewner evolution."} Probability Theory and Related Fields 181.1-3 (2021): 451-488.


  
\bibitem[JL18]{JL18}
M. Jahangoshahi and G. Lawler, \emph{On the smoothness of the
  partition function for multiple {S}chramm-{L}oewner evolutions}, J. Stat.
  Phys. \textbf{173} (2018), no.~5, 1353--1368. \MR{3878346}

\bibitem[MZ24a]{MZ24a}
N. Makarov and J. Zhang, \emph{Multiple radial SLE(0) and classical Calegero-Sutherland system}, Preprint in arXiv:2410.21544, 2024.

\bibitem[MZ24a]{MZ24a}
N. Makarov and J. Zhang, \emph{Multiple radial SLE($\kappa$) and quantum Calegero-Sutherland system}, Preprint in arXiv:2505.14762, 2025

\bibitem[JZ25a]{JZ25a}
J. Zhang, \emph{Multiple chordal SLE(0) and classical Calegero-Moser system}, Preprint in arXiv:2505.17129, 2025

\bibitem[JZ25b]{JZ25b}
J. Zhang, \emph{Multiple chordal SLE($\kappa$) and quantum Calegero-Moser system}, Preprint in arXiv:2505.16093, 2025

\bibitem[JZ25t]{JZ25t}
J. Zhang, \emph{On multiple SLE systems and their deterministic limits}, Ph. D. thesis, California Institute of Technology, 2025.


\bibitem[KL07]{KL07}
M. Kozdron and G. Lawler, \emph{The configurational measure on
  mutually avoiding {SLE} paths}, Universality and renormalization, Fields
  Inst. Commun., vol.~50, Amer. Math. Soc., Providence, RI, 2007, pp.~199--224.
  \MR{2310306}

\bibitem[KM13]{KM13}
N-G Kang and N. Makarov, \emph{Gaussian free field and conformal
  field theory}, Ast\'erisque \textbf{353} (2013), viii+136. \MR{3052311}

\bibitem[KM17]{KM17}
\bysame, \emph{Calculus of conformal fields on a compact Riemann surface},
  arXiv:1708.07361 [math-ph] (2017).

\bibitem[KM21]{KM21}
\bysame, \emph{Conformal field theory on the {R}iemann sphere and its boundary
  version for {SLE}}, arXiv:2111.10057 [math-ph (2021).

\bibitem[KP16]{KP16}
K Kyt\"{o}l\"{a} and E. Peltola, \emph{Pure partition functions of
  multiple {SLE}s}, Comm. Math. Phys. \textbf{346} (2016), no.~1, 237--292.
  \MR{3528421}



\bibitem[Law04]{Lawler:trieste}
G. Lawler, \emph{Conformally invariant processes in the plane}, School
  and {C}onference on {P}robability {T}heory, ICTP Lect. Notes, XVII, Abdus
  Salam Int. Cent. Theoret. Phys., Trieste, 2004, pp.~305--351. \MR{2198851}

\bibitem[Law05]{Lawler:book}
\bysame, \emph{Conformally invariant processes in the plane}, Mathematical
  Surveys and Monographs, vol. 114, American Mathematical Society, Providence,
  RI, 2005. \MR{2129588}

\bibitem[Law09a]{Law09a}
\bysame, \emph{Schramm-{L}oewner evolution ({SLE})}, Statistical mechanics,
  IAS/Park City Math. Ser., vol.~16, Amer. Math. Soc., Providence, RI, 2009,
  pp.~231--295. \MR{2523461}

\bibitem[Law09b]{Law09b}
\bysame, \emph{Partition functions, loop measure, and versions of
  {SLE}}, J. Stat. Phys. \textbf{134} (2009), no.~5-6, 813--837. \MR{2518970}

\bibitem[Lax68]{Lax68}
P. Lax, \emph{Integrals of nonlinear equations of evolution and solitary
  waves}, Comm. Pure Appl. Math. \textbf{21} (1968), 467--490. \MR{235310}

\bibitem[LSW02]{LSW02}
G. Lawler, O. Schramm, W. Werner, \emph{One-arm exponent for critical 2D percolation}.
Electronic Journal of Probability, Electron. J. Probab. 7(none), 1-13, (2002)

\bibitem[LSW03]{LSW:restriction}
G. Lawler, O. Schramm, and Wendelin Werner, \emph{Conformal restriction:
  the chordal case}, J. Amer. Math. Soc. \textbf{16} (2003), no.~4, 917--955.
  \MR{1992830}

\bibitem[LSW04]{LSW04}
\bysame,
\emph{Conformal invariance of planar loop erased random walks and uniform spanning trees}, Ann. Probab., 32(1B):939-995, 2004.


\bibitem[Mos75]{Mos75}
J.~Moser, \emph{Three integrable {H}amiltonian systems connected with
  isospectral deformations}, Advances in Math. \textbf{16} (1975), 197--220.
  \MR{375869}


\bibitem[MS16a]{MS16:imaginary1}
J. Miller and S. Sheffield, \emph{Imaginary geometry {I}: interacting
  {SLE}s}, Probab. Theory Related Fields \textbf{164} (2016), no.~3-4,
  553--705. \MR{3477777}

\bibitem[MS16b]{MS16:imaginary2}
\bysame, \emph{Imaginary geometry {II}: reversibility of {${\rm
  SLE}_\kappa(\rho_1;\rho_2)$} for {$\kappa\in(0,4)$}}, Ann. Probab.
  \textbf{44} (2016), no.~3, 1647--1722. \MR{3502592}

\bibitem[MS16c]{MS16:imaginary3}
\bysame, \emph{Imaginary geometry {III}: reversibility of {$\rm SLE_\kappa$}
  for {$\kappa\in(4,8)$}}, Ann. of Math. (2) \textbf{184} (2016), no.~2,
  455--486. \MR{3548530}

\bibitem[MS17]{MS16:imaginary4}
\bysame, \emph{Imaginary geometry {IV}: interior rays, whole-plane
  reversibility, and space-filling trees}, Probab. Theory Related Fields
  \textbf{169} (2017), no.~3-4, 729--869. \MR{3719057}

\bibitem[MTV09]{MTV:Shapiro}
E. Mukhin, V. Tarasov, and A. Varchenko, \emph{The B. and M.
  Shapiro conjecture in real algebraic geometry and the Bethe ansatz}, Ann.
  of Math. (2) \textbf{170} (2009), no.~2, 863--881. \MR{2552110}
  

\bibitem[Pel19]{Pel19} E. Peltola, \emph{Towards a conformal field theory for Schramm-Loewner evolutions},
J. Math. Phys., 60(10):103305, 39, 2019.
\bibitem[Pel20]{Pel20}  E. Peltola, \emph{Basis for solutions of the Benoit \& Saint-Aubin PDEs with particular asymptotics properties}, Ann. Inst. Henri Poincaré D, 7(1):1-73, 2020.

\bibitem[PW19]{PW19}
E. Peltola and H. Wu, \emph{Global and local multiple {SLE}s for
  {$\kappa \leq 4$} and connection probabilities for level lines of GFF},
  Comm. Math. Phys. \textbf{366} (2019), no.~2, 469--536. \MR{3922531}

\bibitem[PW20]{PW20}
E. Peltola and Y. Wang, \emph{Large deviations of multichordal
  {SLE}$_{0+}$, real rational functions, and zeta-regularized determinants of
  {L}aplacians}, arXiv:2006.08574, to appear in J. Eur. Math. Soc.

\bibitem[RS14]{RS14}
D. Ridout and Y. Saint-Aubin, \emph{Standard Modules, Induction, and the Temperley-Lieb Algebra}, arXiv preprint: 1204.4505, 2014.

\bibitem[S85]{S85}
S, Lang,  \emph{$SL_2(\mathbb{R})$}, Graduate texts in mathematics, 105, Springer, 1985

\bibitem[S02a]{S02a}
I. Scherbak, \emph{Asymptotic solutions to the $sl_2$ KZ equation and the intersection of Schubert classes}, Preprint in arXiv:math/0207218
, 2002.

\bibitem[S02b]{S02b}
\bysame, \emph{Rational Functions with Prescribed Critical Points}. GAFA, Geom. funct. anal. 12, 1365–1380 (2002). 

\bibitem[SV03]{SV03}
I. Scherbak and A. Varchenko, \emph{Critical points of functions, $sl_2$ representations, and Fuchsian differential equations 
with only univalued solutions}, Mosc. Math. J., 2003, Volume 3, Number 2, Pages 621–645.


\bibitem[Sch00]{Sch00}
O. Schramm. \emph{Scaling limits of loop-erased random walks and uniform spanning trees}, Israel J. Math., 118(1):221-288, 2000.

\bibitem[Sch07]{Sch07}
\bysame, \emph{Conformally invariant scaling limits: an overview and a collection of problems}, In International Congress of Mathematicians, vol. 1, pp. 513-543. Eur. Math. Soc., Zürich, 2006.

\bibitem[SS54]{SS54}
M. Schiffer and D. C. Spencer. \emph{Functionals of Finite Riemann Surfaces}, Princeton University Press, 1954 .


\bibitem[Smi06]{Smi06}
S. Smirnov. \emph{Towards conformal invariance of $2 \mathrm{D}$ lattice models}, In International Congress of Mathematicians, vol. 2, pp. 1421-1451. Eur. Math. Soc., Zürich, 2006.

\bibitem[SS09]{SS09}
O. Schramm and S. Sheffield, \emph{Contour lines of the two-dimensional
  discrete {G}aussian free field}, Acta Math. \textbf{202} (2009), no.~1,
  21--137. \MR{2486487}



\bibitem[SW05]{SW05}
O. Schramm and D. Wilson, \emph{S{LE} coordinate changes}, New York J.
  Math. \textbf{11} (2005), 659--669. \MR{2188260}

\end{thebibliography}
\end{document}